\documentclass[a4paper,12pt]{article}
\usepackage[dvipsnames,table,svgnames]{xcolor}
\usepackage[utf8]{inputenc}
\usepackage{verbatim}
\usepackage[margin=3cm]{geometry}
\usepackage[algo2e,ruled,vlined]{algorithm2e}
\usepackage{algorithm}
\usepackage{here}
\usepackage{setspace}
\usepackage[normalem]{ulem}
\usepackage{algpseudocode}
\usepackage[T1]{fontenc}
\usepackage{dsfont}
\usepackage{url}
\usepackage[inline]{enumitem}
\usepackage{microtype}
\usepackage[affil-sl]{authblk}
\usepackage{graphicx}
\usepackage{nicematrix}
\usepackage{capt-of}

\usepackage{tikz}
\usetikzlibrary{tikzmark}
\usetikzlibrary{positioning}
\usetikzlibrary{graphs}
\usetikzlibrary{graphs.standard}
\usetikzlibrary{decorations.pathmorphing}
\usetikzlibrary{decorations.text}
\usetikzlibrary{patterns,shadings}
\usetikzlibrary{matrix}
\pgfdeclarelayer{bg}    
\pgfsetlayers{bg,main}  

\usepackage{amsmath}
\usepackage{amssymb}
\usepackage{amsthm}
\usepackage{mathtools}
\usepackage{nicefrac}
\usepackage{amsfonts,latexsym,graphics}
\usepackage{color}
\usepackage{colortbl}
\usepackage[hidelinks]{hyperref}
\usepackage{multirow}
\usepackage[english]{babel}
\usepackage{epsfig}
\usepackage{bm}
\usepackage{makeidx}
\usepackage[labelformat=simple]{subcaption}
\usepackage{adjustbox}
\usepackage{diagbox}

\usepackage{xspace}         

\makeatletter
\let\c@figure\c@table
\let\ftype@figure\ftype@table
\let\ext@figure\ext@table

\makeatother

\title{A unified framework for existing and new constructions of Neumaier graphs}

\author{\qquad Aida Abiad\thanks{\texttt{a.abiad.monge@tue.nl}, Department of Mathematics and Computer Science, Eindhoven University of Technology, The Netherlands\newline Department of Mathematics and Data Science of Vrije Universiteit Brussel, Belgium}
\quad Wouter Castryck\thanks{\texttt{wouter.castryck@esat.kuleuven.be}, imec-COSIC, Department of Electrical Engineering, KU Leuven, Belgium\\Department of Mathematics: Algebra and Geometry, Ghent University, Flanders, Belgium}
\quad Maarten De Boeck\thanks{\texttt{m.deboeck@slo.nl},  Stichting Leerplanontwikkeling, Netherlands\\ Department of Mathematics, Computer Science and Statistics, Ghent University, Flanders, Belgium} 
\newline Antonina P. Khramova\thanks{\texttt{a.khramova@tue.nl},  Department of Mathematics and Computer Science, Eindhoven University of Technology, The Netherlands} \quad  Thijs van Veluw\thanks{\texttt{thijs.vanveluw@ugent.be},
Department of Mathematics, Computer Science and Statistics, Ghent University, Belgium
\newline
Department of Mathematics and Computer Science, Eindhoven University of Technology, The Netherlands}}

\date{}

\newtheorem{thm}{Theorem}[section]
\newtheorem{prp}[thm]{Proposition}
\newtheorem{lema}[thm]{Lemma}

\theoremstyle{definition}

\newtheorem{exa}[thm]{Example}

\def\a{\alpha}
\def\e{\varepsilon}
\def\N{\mathbb{N}}
\def\Z{\mathbb{Z}}
\def\F{\mathbb{F}}

\def\Fq{\mathbb{F}_q}
\def\Fp{\mathbb{F}_p}
\def\Eis{\mathbb{Z}[\zeta_6]}

\DeclareMathOperator{\Sym}{Sym}
\DeclareMathOperator{\ev}{ev}
\DeclareMathOperator{\im}{im}
\DeclareMathOperator{\Tr}{Tr}
\DeclareMathOperator{\Char}{char}
\DeclareMathOperator{\PG}{PG}
\DeclareMathOperator{\AG}{AG}

\usepackage{etoolbox}

\apptocmd{\thebibliography}{%
  \setlength{\itemsep}{0pt}
  \setlength{\parsep}{0pt}
  \setlength{\parskip}{0pt}
}{}{}

\begin{document}

\date{}

\maketitle

\begin{abstract}
A Neumaier graph is a non-complete edge-regular graph containing a regular clique; it is called strictly Neumaier if it is not strongly regular. In this paper we present a construction using finite rings that unifies several known results and yields three new families, each containing infinitely many strictly Neumaier graphs. We characterize when the resulting graphs are strongly regular. The family obtained from norm functions of finite field extensions includes the collinearity graphs of generalized quadrangles of \(*\)-Tits type arising from the classical hyperoval. We further show that, when the cyclotomic numbers involved have order at most seven, four types of rings suffice to obtain, up to isomorphism, all strictly Neumaier graphs arising from our construction. As a consequence of our results, each of these four types has now been used to construct Neumaier graphs.
\end{abstract}

\section{Introduction}

A non-empty regular graph is called \emph{edge-regular} if any two adjacent vertices have the same number of common neighbors. An edge-regular graph is \emph{strongly regular} if, additionally, any two non-adjacent vertices also have the same number of common neighbors. A \emph{regular clique} in a regular graph is a clique (i.e.\ a set of vertices that are pairwise adjacent) having the property that every vertex outside of it is adjacent to the same positive number of vertices of the clique, denoted by $e$ (referred to as \emph{nexus}). A \emph{Neumaier graph} is a non-complete edge-regular graph containing a regular clique. A Neumaier graph that is not a strongly regular graph is called a \emph{strictly Neumaier graph}; equivalently, a strictly Neumaier graph is a Neumaier graph whose complement is not edge-regular.

In his 1981 paper \cite{Neumaier}, Neumaier studied regular cliques in edge-regular graphs, and he showed that all vertex-transitive, edge-transitive graphs with a regular clique are strongly regular. He subsequently raised the question whether all edge-regular graphs with a regular clique are strongly regular, i.e.\ whether there are strictly Neumaier graphs. This question was answered in 2018 by Greaves and Koolen \cite{GK2018} by providing infinitely many strictly Neumaier graphs, and in recent years this topic has attracted a considerable amount of attention in the literature. Neumaier graphs have witnessed several  applications in algebraic  combinatorics. Strictly Neumaier graphs provide counterexamples
that clarify the limits of spectral results for strongly regular
graphs~\cite{ADDK2022}. They also provide non-isomorphic graphs with the same
spectrum, illustrating the limitations of spectral methods for graph characterizations~\cite{DEGK2026+,EGKM2023}. Their connections with block designs help with the classification and nonexistence proofs of Neumaier graphs~\cite{ADZ2024},
while further constructions, like the ones mentioned above, provide examples for studying algebraic
structures associated with graphs and related to graph
isomorphism~\cite{GT2026+}.

Numerous constructions of strictly Neumaier graphs have been proposed using cyclotomy \cite{GT2026,GK2018}, switching methods \cite{ADZ2024,EGP2019,EGKM2023}, Cayley graphs \cite{EGRZ2026+,Jazaeri}, and other methods \cite{DEGK2026+,GT2026+}. Furthermore, a prolific method of constructing Neumaier graphs with a partition of their vertex sets into regular cliques (the so-called \emph{spread}) has been proposed by Evans, Goryainov, Konstantinova, and Mednykh \cite{EGKM2023} as a generalization of the results from \cite{GK2019}, which has been used in combination with Jacobi sums to construct infinitely many strictly Neumaier graphs by Abiad, Castryck, De Boeck, Koolen, and Zeijlemaker in \cite{ACDKZ2023}. Moreover, necessary conditions or non-existence results for this graph class have also been studied, see \cite{ACDKZ2023,ADZ2024,ADDK2022,EGP2019,Jazaeri,YK2025}. For an overview of strictly Neumaier graph parameter sets of small order, see \cite[Table 1]{ADZ2024}\footnote{We should note that in this table, the parameter set $(58,30,11;2,4)$ is missing, also see the older \cite[Table 1]{ACDKZ2023}.}.

In this paper, we use finite rings to construct edge-regular graphs with a partition of their vertex sets into regular cocliques of nexus 1 as required in the construction of \cite{EGKM2023}. This more general algebraic framework to construct Neumaier graphs allows us to unify and recover the very first construction \cite{GK2018}, the construction from \cite{ACDKZ2023}, and a special case of the construction from \cite{GT2026}. In addition to these existing constructions, we also use the proposed framework to obtain three new families of Neumaier graphs. These are based on the norm functions of extensions of finite fields, the ring of dual numbers over $\F_4$, and quotient rings of the Eisenstein integers, respectively. Furthermore, we prove that each of these three new families gives infinitely many strictly Neumaier graphs. For this we use several number-theoretic results ranging from the Ax-Katz theorem, formulas for the sextic cyclotomic numbers, to Hensel's lemma. The first of these three new families includes the collinearity graphs of generalized quadrangles of $\ast$-Tits type (cf.\ \cite[Section 3.7.2]{VMaldeghem}).

We furthermore show that, under an upper bound of 7 on the size of the image of the defining group homomorphism (implying an upper bound of 7 of the order of the cyclotomic numbers involved), only four types of finite rings suffice to obtain, up to graph isomorphism, all strictly Neumaier graphs produced by our construction. Neumaier graphs arising from the first two types of rings are provided by \cite{ACDKZ2023,GK2018} and by our first new family of Neumaier graphs. The other two types of rings precisely correspond to our second and third new infinite families of Neumaier graphs, implying that our contributions are in essence exhaustive under the abovementioned restrictions.

This paper is organized as follows. In Section~\ref{sec:preliminaries}, we set out the necessary preliminaries that are needed for our results and proofs. In Section~\ref{sec:construction}, we present our main result (Theorem~\ref{thm:general}), and show how it provides a unifying framework for the aforementioned existing constructions of Neumaier graphs. In Section~\ref{sec:norm}, we present the first of the new families using the norm function of finite field extensions. We also prove its infinitude. In Section~\ref{sec:even more constructions}, we investigate the types of rings that are suitable to our framework of constructing Neumaier graphs, and provide two new infinite families of strictly Neumaier graphs accordingly.

\section{Preliminaries}\label{sec:preliminaries}
In this section we provide the graph and ring theoretic notation and background used throughout the paper. Moreover, we recall the machinery developed by \cite{EvansThesis,EGKM2023} to construct nexus 1 Neumaier graphs with a spread. Lastly, we recall some definitions and results on cyclotomic numbers, Eisenstein integers, and conics over finite fields that are needed in our results and proofs.
\subsection{Basic graph and ring theory}
We consider simple, undirected, loopless graphs $\Gamma$, consisting of a vertex set $V(\Gamma)$ and an edge set $E(\Gamma)$ such that each $e\in E(\Gamma)$ satisfies $e\subseteq V(\Gamma)$ and $|e|=2$. If for two vertices $x,y\in V(\Gamma)$ the set $\{x,y\} \in E(\Gamma)$, then we say that $x$ and $y$ are \emph{adjacent} or that $x$ and $y$ are \emph{neighbors}, and write $x\sim_\Gamma y$ (or $x\sim y$ if the graph is clear from context).

We denote the parameters of an edge-regular graph by $(v,k,\lambda)$, where $v$ is the number of vertices, $k$ is the valency of the graph, and every pair of adjacent vertices has exactly $\lambda$ common neighbors. For a strongly regular graph, we use the tuple $(v,k,\lambda,\mu)$, where, additionally, every pair of distinct non-adjacent vertices has exactly $\mu$ common neighbors. For a Neumaier graph, we use the tuple $(v,k,\lambda;e,s)$, where $(v,k,\lambda)$ are its edge-regular parameters, $e$ denotes the nexus, and $s$ is the size of the regular clique.

Similar to regular cliques, a \emph{regular coclique of nexus} $e$ in a graph is a coclique (i.e.\ a set of vertices pairwise non-adjacent) such that every vertex outside of the coclique has exactly $e$ neighbors in the coclique.

All rings $R$ are assumed to be associative and to have a multiplicative identity $1_R$ that is distinct from $0_R$ (where we leave off the subscripts if the ring is known from context). Accordingly, ring homomorphisms are assumed to map the multiplicative identity of the domain ring to the multiplicative identity of the codomain ring. For a ring $R$, we write $(R,+)$ for the underlying abelian group under addition, and $R^*$ for the multiplicative group of units. For rings $R$ and $S$ and for groups $G$ and $H$ we write their direct products as $R\times S$ and $G\times H$. Note that $(R\times S, +)\cong (R, +) \times (S,+)$ and that $(R\times S)^*\cong R^* \times S^*$.

For a finite ring $R$, the \emph{characteristic} of $R$, denoted $\Char(R)$, is the smallest positive integer $c$ such that $\underbrace{1+\dots+1}_{c \text{ times}}=0$.

For a positive integer $n$, we use $\Z_n:=\Z/n\Z$, and for integers $a$ and $b$, we write $a \equiv_n b$ if $a-b$ is divisible by $n$. For a prime power $q$ we write $\Fq$ for the finite field with $q$ elements; if $q=p^e$ for a prime number $p$ and positive integer $e$, then $\Char(\Fq)=p$. For a commutative ring $R$, we write $R[x]$ for the ring of polynomials in the variable $x$, and furthermore for elements $r_1,\dots,r_n \in R$, the \emph{ideal of $R$ generated by $r_1,\dots,r_n$}, denoted $(r_1,\dots,r_n)$ refers to the set $\{r_1s_1+r_2s_2+\dots+r_ks_k : s_1,\dots,s_k \in R\}$.

\subsection{Cayley graphs}

Let $G$ be an abelian group with operation $+$ and identity element $0$, and let $S$ be a subset of $G$ not containing $0$ that is inverse-closed (i.e.\ for all $s\in S$ also $-s\in S$). Then the \emph{Cayley graph on $G$ with connecting set $S$} is defined as the graph $\Gamma(G,S)$ with the elements of $G$ as vertices, such that $g\sim h$ whenever $g-h\in S$. Note that since $S$ does not contain $0$ and is inverse-closed, $\Gamma(G,S)$ is a loopless undirected graph.

Cayley graphs are highly symmetric. For example, for every $g\in G$ the mapping $\a_g: h \mapsto g+h$ is a graph automorphism, making $\Gamma(G,S)$ vertex-transitive (i.e.\ for every pair of vertices $u,v$ there is a graph automorphism sending $u$ to $v$). Furthermore, if $\varphi$ is a group automorphism of $G$ that fixes $S$ set-wise, then $\varphi$ induces a graph automorphism of $\Gamma(G,S)$. In particular, if $G$ is the additive group of a ring $R$, and $r\in R^*$ such that $r\cdot S=S$, then $\lambda_r : g\mapsto rg$ induces a graph automorphism of $\Gamma(G,S)$. This implies the following lemma, which we use in our arguments.

\begin{lema}\label{lema:edge-transitive}
    Let $R$ be a ring, and let $S$ be a subgroup of $R^*$ containing $-1_R$. Then $\Gamma((R,+),S)$ is edge-transitive, i.e.\ for every pair of edges $\{u,v\},\{u',v'\}$ there is a graph automorphism $\varphi$ with $\{\varphi(u),\varphi(v)\}=\{u',v'\}$, namely $\varphi=\a_{u'}\circ \lambda_{(v'-u')(v-u)^{-1}} \circ \a_{-u}$.
\end{lema}

\subsection{Nexus 1 Neumaier graphs with a spread}
A \emph{spread} in a Neumaier graph $\Gamma$ is a partition of the vertex set of $\Gamma$ into regular cliques. In \cite{EvansThesis,EGKM2023}, Evans, Goryainov, Konstantinova, and Mednykh developed a method to obtain nexus 1 Neumaier graphs with a spread. This construction is a generalization of the construction by Greaves and Koolen from \cite{GK2019}.

For $i=1,\dots, t$, let $\Gamma_i$ be an edge-regular graph with parameters $(v,k,\lambda)$, such that $C_{i,1},\dots,C_{i,k+1}$ is a partition of its vertex set into 1-regular cocliques (i.e.\ cocliques $C$ such that every vertex outside of $C$ is adjacent to exactly one vertex of $C$). Let $\pi_1=(),\pi_2,\dots,\pi_t$ be $t$ permutations in $\Sym_{k+1}$. Then the graph $F_{(\pi_2,\dots,\pi_t)}(\Gamma_1,\dots,\Gamma_t)$ is defined as follows. Start with the disjoint union of the graphs $\Gamma_1,\dots,\Gamma_t$, and draw new edges between any vertex $u$ from $C_{i,\pi_i(j)}$ and any vertex from $C_{i',\pi_{i'}(j)}$ that is distinct from $u$, for every $i,i'\in \{1,\dots, t\}$ (not necessarily distinct) and every $j\in \{1,\dots, k+1\}$.

\begin{thm}[\cite{EvansThesis,EGKM2023}]\label{thm:EvansMachine} Let $\Gamma_i$, $C_{i,j}$ and $\pi_i$ be as above. If $t=\frac{(\lambda+2)(k+1)}{v}$, then $F_{(\pi_2,\dots,\pi_t)}(\Gamma_1,\dots,\Gamma_t)$ is a Neumaier graph with parameters $(vt,k+\lambda+1,\lambda;1,\lambda+2)$.
\end{thm}

The following theorem gives a sufficient condition to ensure that the Neumaier graph constructed using the method from Theorem~\ref{thm:EvansMachine} is a strictly Neumaier graph. Note that we include the case where the $\Gamma_i$'s are disconnected; if each of the $\Gamma_i$'s is a disjoint union of complete graphs, then $F_{(\pi_2,\dots,\pi_t)}(\Gamma_1,\dots,\Gamma_t)$ is strongly regular.

\begin{thm}[cf.\ \cite{ACDKZ2023,EGKM2023}\label{thm:EMstrict}] Let $\Gamma_i$, $C_{i,j}$ and $\pi_i$ be as above, and suppose $t=\frac{(\lambda+2)(k+1)}{v}$. Suppose that the $\Gamma_i$'s are not disjoint unions of complete graphs. If $t\ge 2$, then $F_{(\pi_2,\dots,\pi_t)}(\Gamma_1,\dots,\Gamma_t)$ is a strictly Neumaier graph.
\end{thm}

\subsection{Cyclotomic numbers}
Cyclotomic numbers were used in the first construction of Neumaier graphs by Greaves and Koolen \cite{GK2018}. In Lemma~\ref{lema:b cyclotomy}, we will use cyclotomic numbers in a similar way, which is why we recall their definition and some results here.

If $G$ is a cyclic multiplicative group generated by $g\in G$, then we write $\log_g$ for the group homomorphism from $G$ to $\Z_{|G|}$ defined by $\log_g(g^i)=i$. Now let $q$ be a prime power, write $q=mn+1$ for some positive integers $m$ and $n$, and let $\alpha$ be a primitive element of $\F_q$. Then for $i,j\in \Z$ consider the set
$$C_q^m(\a;i,j)=\{x\in \F_q^*\setminus \{1\}: \log_\alpha(x-1)\equiv_m i, \log_\alpha(x)\equiv_m j\}.$$
The \emph{cyclotomic number} $c_q^m(\alpha;i,j)$ is defined as the size of $C^m_q(\a;i,j)$. Note that the integers $i$ and $j$ can be taken modulo $m$. The cyclotomic numbers satisfy the following properties, which we need for our results.

\begin{prp}[{\cite{BerndtEvansWilliams, Storer}}]\label{prp:cyclo} Let $q$ be a prime power, let $m$ and $n$ be integers such that $q=mn+1$, and let $\alpha$ be a primitive element of $\Fq$. Then the following hold:
    \begin{enumerate}[label=(\arabic*)]
    \item for every $i,j\in \Z$ we have $c_q^m(\a;i,j)=c_q^m(\a;-i,j-i)$;
    \item for every $i,j\in \Z$ we have $c_q^m(\a;i,j)=c_q^m(\a; j+ \log_\a(-1),i+\log_\a(-1))$,
    where
    \[
\log_\a(-1)\equiv_m
\begin{cases}
    0 &$if $qn$ is even,$\\
    m/2 &$if $qn$ is odd;$
\end{cases}
\]
    \item for every $j\in \Z$ we have
    \[\sum_{i=0}^{m-1} c_q^m(\a;i,j)=
    \begin{cases}
        n &$if $j\not \equiv_m 0,\\
        n-1 &$if $j\equiv_m 0;
    \end{cases}
    \]
    \item for every $i\in \Z$, the cyclotomic number $c_q^m(\a;0,i)$ is odd if and only if $q$ is odd and $\log_\a(2) \equiv_m i$.
    \end{enumerate}
\end{prp}

We will need the following explicit formulas for specific sextic cyclotomic numbers. Although these formulas follow from results in \cite{BerndtEvansWilliams} relating cyclotomic numbers to Jacobi sums, to our knowledge these formulas have not been written down explicitly, so we include the proof for completeness.
\begin{lema}[cf.\ \cite{BerndtEvansWilliams}]\label{lema:formulas specific cyclotomic numbers}
    Let $x$ and $y$ be integers with $x\equiv_3 2$ and $y\equiv_3 0$ such that $p=x^2+3y^2$ is a prime number. Let $\a$ be a primitive element of $\F_p$. Then the following hold.
    \begin{enumerate}[label=(\arabic*)]
        \item If $p\equiv_{12} 1$, then $36c_p^6(\a;2,4)=(x+1)^2+3y^2$,
        \item If $p\equiv_{12} 7$, then $36c_p^6(\a;2,1)=(x+1)^2+3y^2$.
    \end{enumerate}
\end{lema}
\begin{proof}
    We only show the second part; the first part follows similarly. By \cite[Theorem 2.5.1]{BerndtEvansWilliams} we have
    $$36c_p^6(\a;2,1)=\sum_{u=0}^5 \sum_{v=0}^5 (-1)^u \zeta_6 ^{-2u-v} J^\ast(u,v)=\sum_{u=0}^5\sum_{v=0}^5 \zeta_6^{u-v} J^*(u,v),$$
    where $\zeta_6$ is a primitive sixth root of unity, and where $J^*$ denotes a variation of the Jacobi sums, whose values are given in \cite[Section 3.1]{BerndtEvansWilliams}: we have
    \[
    J^*(u,v)=\begin{cases}
        p-2 &$if $u=v=0;\\
        -1 &$if exactly one of $u,v$ is zero$;\\
        -(-1)^{uv} &$if $u+v=6;\\
        (-1)^{uv}(x+\e iy\sqrt 3) &$if $u\ne0,v\ne0,u+v<6;\\
        (-1)^{uv}(x-\e iy\sqrt 3) &$if $u+v>6.
    \end{cases}
    \]
    where the choice of $\e \in \{1,-1\}$ accounts for the sign ambiguity of $y$, see \cite[Section 3.1]{BerndtEvansWilliams} for details. Now the sub-sums on these five cases of the double sum of $36c_p^6(\a;2,1)$ written above evaluate to $p-2$, $2$, $1$, $x+\e iy\sqrt 3$ and $x-\e iy\sqrt 3$ respectively. Hence
     \begin{align*}
    36c_p^6(\a;2,1)&=(p-2)+2+1+(x+\e iy\sqrt 3)+(x-\e iy\sqrt 3)\\
    &=p+2x+1 =(x+1)^2+3y^2,
    \end{align*}
     concluding the proof.
\end{proof}
Similar formulas are possible for the case where $y$ is not divisible by 3, cf.~\cite{BerndtEvansWilliams}.

\subsection{Eisenstein integers}
We recall some definitions and results on Eisenstein integers that are needed throughout Section~\ref{sec:even more constructions}. The ring of \emph{Eisenstein integers} $\Eis$ is the smallest subring of the ring of complex numbers that contains a  primitive sixth root of unity $\zeta_6$. Alternatively, $\Eis$ is isomorphic to the quotient $\Z[x]/(x^2-x+1)$ of the polynomial ring over the integers. There are six units in $\Eis$, namely $1$, $\zeta_6$, $\zeta_6^2=\zeta_6-1$, $\zeta_6^3=-1$, $\zeta_6^4=-\zeta_6$, and $\zeta_6^5=1-\zeta_6$. The Eisenstein integers have one non-trivial ring automorphism, namely conjugation (where the conjugate of $\zeta_6$ is $\zeta_6^5=1-\zeta_6$). The \emph{norm} function on the Eisenstein integers is denoted $N:\Eis \to \Z_{\ge 0}$ and defined by
$$N(a\zeta_6+b)=(a\zeta_6+b)(a(1-\zeta_6)+b)=a^2+ab+b^2.$$
We have the following.
\begin{prp}[see {\cite[Chapter 9]{IrelandRosen}}]\label{prop:eis}
    \begin{enumerate}[label=(\arabic*)]
    \item The ring of Eisenstein integers is a principal ideal domain. In other words, every ideal $I$ of $\Eis$ is generated by a single element $a\zeta_6+b$ (which, by multiplying with an appropriate unit, can be chosen such that $a \ge 0$ and $b \ge 0$).
    \item If $a\zeta_6+b\ne 0$, then the index of the ideal generated by $a\zeta_6+b$ is equal to $N(a\zeta_6+b)=a^2+ab+b^2$.
    \item The ring of Eisenstein integers is a unique factorization domain, meaning that every non-zero element admits a unique factorization into primes (up to a unit factor and the ordering of the primes). There are three types of Eisenstein primes (up to a unit factor): firstly, the primes $p\in \N$ with $p\equiv_3 2$; secondly, the Eisenstein integers $a\zeta_6+b$ for which $N(a\zeta_6+b)$ is a prime $p\in \N$ with $p\equiv _3 1$; and thirdly, the Eisenstein integer $\zeta_6+1$.
    \item In particular, every prime $p\in \N$ with $p\not \equiv_3 2$ can be written as $p=a^2+ab+b^2$ for some $a,b\in \N$. More generally, a positive integer $r$ can be written as $a^2+ab+b^2$ if and only if all prime divisors $p \mid r$ with $p\equiv_3 2$ divide $r$ an even number of times.
    \end{enumerate}
\end{prp}

We also need the following result, which is the Eisenstein analogue of Dirichlet's Theorem on primes in arithmetic progressions. Similar to the integer setting, for $k\in \Z[\zeta_6]\setminus \{0\}$ and $z,z'\in \Z[\zeta_6]$ we write $z\equiv_k z'$ if $z-z'$ is contained in the ideal $(k)$. Two Eisenstein integers are called coprime if their only common divisors are units.
\begin{thm}[Eisenstein primes in arithmetic progressions (see {\cite[Section V.6]{Neukirch}})]\label{thm:Dirichlet}
    Let $k\in \Eis\setminus \{0\}$ and $z\in \Eis$ such that $k$ and $z$ are coprime. Then there are infinitely many Eisenstein primes $\pi$ with $\pi\equiv_k z$.
\end{thm}

\subsection{Conics over finite fields and Hensel's lemma}

We fix the terminology and collect the results from the literature that are needed for the proof of Lemma~\ref{lema:zeta6 infinite}.
For $n\geq1$ we write $\PG(n,q)$ for the projective space of
dimension $n$ over $\mathbb F_{q}$: its points are the $1$-dimensional
subspaces of $\mathbb F_{q}^{\,n+1}$, and we write
$(x_{0}:x_{1}:\dots:x_{n})$ for the point spanned by a non-zero vector
$(x_{0},x_{1},\dots,x_{n})\in\mathbb F_{q}^{\,n+1}$. Thus
$(x_{0}:\dots:x_{n})=(\lambda x_{0}:\dots:\lambda x_{n})$ for every
$\lambda\in\mathbb F_{q}^{*}$, and each point has exactly $q-1$ such
representatives. A point with $x_n\neq0$ is called \emph{affine} and has a unique representative of
the form $(x_0:x_1: \dots : x_{n-1} :1)$. The points with $x_n=0$ are the points \emph{at
infinity}. The \emph{affine space} $\AG(n,q)$ is now defined as the set of affine points of $\PG(n,q)$.

A \emph{conic} $\mathcal C$ of $\PG(2,q)$ is the zero set of a non-zero quadratic form $F\in \Fq[X,Y,Z]$ (i.e.\ a homogeneous polynomial of degree 2); as $F$ is homogeneous, we can consider its set of zeros in $\PG(2,q)$.

If $f\in\mathbb F_{q}[x,y]$ has degree $2$, then its \emph{homogenization}
$$F(X,Y,Z):=Z^{2}\cdot f(X/Z,Y/Z)$$
is a quadratic form, and $(x,y)\mapsto(x:y:1)$ is a
bijection from the set of solutions of $f$ in $\mathbb F_{q}^{2}$ onto the set of affine points of the conic associated with $F$.

\begin{prp}[{cf.\ \cite[Chapter 7]{Casse}}]\label{prop:conic count}
    Consider a quadratic form $F\in \F_q[X,Y,Z]$ and the conic $\mathcal C : F(X,Y,Z)=0$ of $\PG(2,q)$. If no point $P\in \mathcal C$ satisfies
    \[
    \begin{cases}
        \frac{\partial F}{\partial X}(P)=0,\\
        \frac{\partial F}{\partial Y}(P)=0,\\
        \frac{\partial F}{\partial Z}(P)=0,\\
    \end{cases}
    \]
    then $|\mathcal C|=q+1$.
\end{prp}

\begin{thm}[Hensel's lemma, cf.\ \cite{Serre}]\label{thm:Hensel}
    Let $p$ be a prime, let $f\in\mathbb Z[x_{1},\dots,x_{n}]$ and let
$a\in\mathbb Z^{n}$. Assume that
\[
  f(a)\equiv_{p}0
  \qquad\text{and}\qquad
  \frac{\partial f}{\partial x_{j}}(a)\not\equiv_{p}0
  \ \text{ for some }j.
\]
Then for every $e\geq1$ there exists $b\in\mathbb Z^{n}$ with $b_i\equiv_{p}a_i$ for all $i=1,\dots,n$ and $f(b)\equiv_{p^{e}}0$.
\end{thm}

\begin{thm}[Hensel's lemma, general form {\cite[Chapter II, Section 2.2, Theorem 1]{Serre}}]\label{thm:Hensel general}
    Let $p$ be a prime, let $f\in\mathbb Z[x_{1},\dots,x_{n}]$, let
$a\in\mathbb Z^{n}$ and let $N,k\in\mathbb Z$ with $0\leq2k<N$. Assume that
\[
  f(a)\equiv_{p^{N}}0
  \qquad\text{and}\qquad
  v_{p}\!\left(\tfrac{\partial f}{\partial x_{j}}(a)\right)=k
  \ \text{ for some }j,
\]
where $v_p$ denotes the $p$-adic valuation (i.e.\ the largest exponent $k$ such that $p^k$ divides the argument). Then for every $e\geq1$ there exists $b\in\mathbb Z^{n}$ with
$b_i\equiv_{p^{N-k}}a_i$ for all $i=1,\dots,n$ and $f(b)\equiv_{p^{e}}0$.
\end{thm}

Theorem~\ref{thm:Hensel} is the special case of Theorem~\ref{thm:Hensel general} where $N=1$ and $k=0$.

\section{A unifying framework}\label{sec:construction}

In this section we present our main result (Theorem~\ref{thm:general}). We provide a general ring-theoretic method to construct edge-regular graphs whose vertex sets partition into regular cocliques of nexus 1, from which Theorem~\ref{thm:EvansMachine} constructs Neumaier graphs (if the divisibility requirement of Theorem~\ref{thm:EvansMachine} is satisfied). We also give a sufficient condition for the resulting Neumaier graph to be a strictly Neumaier graph. We then state and prove two auxiliary lemmas that concern the necessary conditions of Theorem~\ref{thm:general}. For instance, in Lemma~\ref{lema:b cyclotomy}, the divisibility condition of Theorem~\ref{thm:general} (coming directly from the divisibility condition in Theorem~\ref{thm:EvansMachine}) is expressed in terms of cyclotomic numbers. Lastly, we show how several existing constructions of Neumaier graphs \cite{ACDKZ2023, GK2018, GT2026} are included in or related to Theorem~\ref{thm:general}, showing that it provides a unifying framework for these existing results.

\subsection{A general construction}

After stating and proving our main result, we prove two auxiliary lemmas concerning the conditions of Theorem~\ref{thm:general}, making it easier to apply.

\begin{thm}\label{thm:general}
    Let $R$ be a finite ring, $q$ be a prime power, and $\xi $ a group homomorphism from $\Fq^*$ to $R^*$ such that $\xi(-1_{\Fq})=-1_R$. Extend $\xi$ to $\Fq$ by setting $\xi(0_{\Fq})=0_R$, and define $\psi:\F_q \to R$ by $\psi(x)=\xi(x)-\xi(x-1_{\Fq})$. Write $B=\psi^{-1}(1_R)$ and $b=|B|$. Let $\Gamma(R,q,\xi)$ be the Cayley graph of $(R\times \Fq,+)$ with connecting set $\{(\xi(x),x) : x \in \Fq^*\}$. Then the following hold.
    \begin{enumerate}
        \item The graph $\Gamma(R,q,\xi)$ is an edge-regular graph with parameters $(q|R|,q-1,b-2)$ whose vertex set can be partitioned into $q$ many $1$-regular cocliques of size $|R|$.
        \item If $|R|$ is a divisor of $b$ (say $t=b/|R|$), then for any choice of permutations $\pi_2,\dots,\pi_t \in \Sym_q$ the graph $\overline \Gamma (R,q,\xi) \coloneqq F_{(\pi_2,\dots,\pi_t)}(\Gamma(R,q,\xi),\dots,\Gamma(R,q,\xi))$ is a Neumaier graph with parameters $(bq,b+q-2,b-2;1,b)$.
        \item The graph $\overline \Gamma(R,q,\xi)$ is strongly regular if and only if $t=1$ and $|\psi^{-1}(r)|$ is constant over all $r\in R\setminus \{1_R\}$. In this case $\overline \Gamma(R,q,\xi)$ has parameters $(bq,b+q-2,b-2,\frac{b+q-2}{b-1})$.
        \item If $R$ is not a field or $\xi:\Fq^*\to R^*$ is not surjective, then $\overline \Gamma (R,q,\xi)$ is a strictly Neumaier graph.
    \end{enumerate}
\end{thm}

\begin{proof}
    For the first part, 
    note that the connecting set
    $S := \{(\xi(x),x) : x \in \Fq^*\}$ contains $-1_{R\times\Fq}=(-1_R,-1_{\Fq})$ because $\xi(-1_{\Fq})=-1_R$, and so $\Gamma(R,q,\xi)$ is a well-defined Cayley graph. Because $S$ is a subgroup of $(R\times \Fq)^*$, the graph $\Gamma(R,q,\xi)$ is edge-transitive by Lemma~\ref{lema:edge-transitive}, and hence edge-regular. It is easy to see that $\Gamma(R,q,\xi)$ has $q|R|$ vertices and has valency $q-1$. For the parameter $\lambda$, note that $B$ contains $0_{\Fq}$ and $1_{\Fq}$, and that the common neighbors of $(0_R,0_{\Fq})$ and $(1_R,1_{\Fq})$ are precisely $(\xi(x),x)$ for $x\in B\setminus \{0_{\Fq},1_{\Fq}\}$, so that $\lambda=b-2$. Next, note that $T\coloneqq S\sqcup \{(0_R,0_{\Fq})\}$ forms a transversal for the subgroup $\{(r,0) : r\in R\}$ of $(R\times \Fq,+)$, so that for each $x\in \Fq$, the set $\{(r,x) : r\in R\}$ is a 1-regular coclique.
    
    The second part follows from Theorem~\ref{thm:EvansMachine} and from the parameters given in the first part.

    For the third part, we first claim that if $\overline\Gamma(R,q,\xi)$ is strongly regular, then $t=1$. Indeed, if $t>1$, then by Theorem~\ref{thm:EMstrict} we know that $\Gamma(R,q,\xi)$ is a disjoint union of complete graphs. This implies that $T$ is a subgroup of $(R\times \Fq,+)$, and that $\xi$ is additive. In particular $\xi:\F_q \to R$ is a ring morphism. Since $\Fq$ is a field, $\xi$ is injective, so $|R|\ge q$. Furthermore, since $\xi$ is additive, it is clear that $\psi(x)=1_R$ for every $x\in \Fq$, and that $b=q$, so that $t=b/|R|>1$ contradicts $|R|\ge q$.

    We now assume that $t=1$, and show that in this case $\overline\Gamma(R,q,\xi)$ is strongly regular if and only if $|\psi^{-1}(r)|$ is constant over all $r\in R\setminus \{1_R\}$. Let $u$ and $v$ be two vertices that are not adjacent in $\overline\Gamma(R,q,\xi)$, then their second coordinates are not equal, and we can apply a suitable automorphism $\varphi$ so that $\varphi(u)=(0_R,0_{\Fq})$ and $\varphi(v)=(r,1_{\Fq})$ for some $r\in R\setminus \{1_R\}$. Note that then a neighbor $w$ of $\varphi(u)$ is also a neighbor of $\varphi(v)$ if and only if $w$ is one of $(1_R,1_{\Fq})$ and $(r-1_R,0_{\Fq})$, or $w=(\xi(x),x)$ with $\psi(x)=r$. In particular, the number of common neighbors of $u$ and $v$ in $\overline \Gamma(R,q,\xi)$ is equal to $2+|\psi^{-1}(r)|$. It follows that $\overline\Gamma(R,q,\xi)$ is strongly regular if and only if $|\psi^{-1}(r)|$ is constant over all $r\in R\setminus \{1_R\}$.

    For the last part, note that $|\psi^{-1}(1_R)|=b$ and
    $$\sum_{r\in R} |\psi^{-1}(r)|=|\Fq|=q.$$
    From this and the previous part it follows that if $\overline \Gamma(R,q,\xi)$ is strongly regular, then $|\psi^{-1}(r)|=(q-b)/(b-1)$ for every $r\in R\setminus \{1_R\}$. In particular, this holds for $r=0_R$. Note that $\psi(x)=0_R$ is equivalent to $\xi(1-1/x)=1_R$. If we write $m=|\xi(\Fq^*)|$, then there are exactly $(q-1)/m$ solutions $y$ to $\xi(y)=1_R$. As this includes $y=1$ which is never equal to $1-1/x$, this means that there are exactly $(q-m-1)/m$ solutions $x$ to $\psi(x)=0_R$. In other words, $|\psi^{-1}(0_R)|=(q-m-1)/m$. Now $|\psi^{-1}(0_R)|=(q-b)/(b-1)$ it implies that $m=b-1$. Because $t=1$ by the previous part, we see that $|R|=b$, and hence that $\xi(\F_q^*)=R\setminus \{0_R\}$. We conclude that $\xi:\Fq^*\to R^*$ is surjective and that $R$ is a field. The result now follows.
\end{proof}

To apply Theorem~\ref{thm:general}, we need a group homomorphism $\xi:\F_q^* \to R^*$ with $\xi(-1_{\Fq})=-1_R$. Given a ring $R$ and an element $\rho\in R^*$, the following auxiliary lemma characterizes which finite fields admit such a group homomorphism with image $\langle \rho \rangle _{R^*}$.
\begin{lema}\label{lema:-1}
    Let $R$ be a finite ring, and let $\rho\in R^*$. Let $\ell$ be the smallest positive integer such that $\rho^{\ell}=-1_R$ (in particular, suppose that such an integer exists). Let $q$ be a prime power, and let $\a$ be a primitive element of $\Fq$. Then sending $\a \mapsto \rho$ induces a well-defined group homomorphism $\xi:\Fq^* \to R^*$ such that $\xi(-1_{\Fq})=-1_R$ if and only if one of the following holds:
    \begin{itemize}
        \item $\Char(R)=2$, $\ell $ is odd, and $q\equiv_{\ell} 1$,
        \item $\Char(R)=2$, $\ell$ is even, and $q\equiv_{2\ell} 1$,
        \item $\Char(R)\ge 3$, and $q\equiv_{4\ell} 2\ell+1$.
    \end{itemize}
\end{lema}
\begin{proof}
    First, suppose that the characteristic of $R$ is 2. Then $-1_R=1_R$, and the order of $\rho$ in $R^*$ is $\ell$. Now $\xi$ is well-defined if and only if $q\equiv_\ell 1$. If $q$ is even, then $\xi(-1_{\Fq})=-1_R$ is automatically satisfied. If $q$ is odd, then
    $$\xi(-1_{\Fq})=\xi(\a^{(q-1)/2})=\rho^{(q-1)/2},$$
    so that $\xi(-1_{\Fq})=-1_R$ is equivalent to $q\equiv_{2\ell} 1$.

    Next, suppose the characteristic of $R$ is at least 3. Then $-1_R\ne 1_R$, and the order of $\rho$ in $R^*$ is $2\ell$. Now $\xi$ is well-defined if and only if $q\equiv_{2\ell} 1$. Then $q$ is odd and so, similarly as above, we have $\xi(-1_{\Fq})=-1_R$ if and only if $(q-1)/2 \equiv_{2\ell} \ell$, concluding the proof.
\end{proof}

As a second auxiliary lemma, we give a formula for the parameter $b$ from Theorem~\ref{thm:general} in terms of cyclotomic numbers. This is analogous to \cite[Theorem 5.7]{ACDKZ2023}, where a formula is provided for a parameter analogous to $b$ in terms of Jacobi sums; note that Jacobi sums and cyclotomic numbers are related \cite[Theorem 2.5.1]{BerndtEvansWilliams} (a fact that is also used in the proof of Lemma~\ref{lema:formulas specific cyclotomic numbers}).

\begin{lema}\label{lema:b cyclotomy}
    Let $R$, $q$, $\xi$, $B$, and $b$ be as in Theorem~\ref{thm:general}. Let $\a$ be a primitive element of $\Fq$, and write $\rho=\xi(\a)$ and $G=\langle \rho \rangle_{R^*}=\{\xi(x) : x \in \Fq^*\}$. Let $m=|G|$ be the order of $\rho$ in $R^*$. Then
    $$b=2+\sum_{r\in  G\cap (G+1_R)} c_q^m\Big (\alpha;\log_{\rho} (r-1_R ),\log_{\rho}(r)\Big ).$$
    If $G=R\setminus \{0_R\}$ (i.e.\ if $R$ is a field and $\rho$ is a primitive element of $R$), then this formula simplifies to
    $$b=2+\sum_{j=1}^{m-1} c_q^m\Big (\a;\log_\rho(\rho^j-1_R),j\Big ).$$
\end{lema}
\begin{proof}
    Let $\psi$ be as in Theorem~\ref{thm:general}. Let $x\in \Fq$ such that $\psi(x)=1_R$. If $\xi(x) =0$ or $\xi(x-1_{\Fq})=0$, then $x=0$ or $x=1$. Otherwise, both $\xi(x)$ and $\xi(x-1_{\Fq})=\xi(x)-1_R$ are elements of $G$, and so $r=\xi(x)$ is an element of $G\cap (G+1_R)$.

    Now let $r\in G\cap (G+1_R)$. Write $B_r$ for the set of elements $x\in \Fq^* \setminus \{1_{\Fq}\}$ such that $\xi(x)=r$ and $\xi(x-1_{\Fq})=r-1_R$. For an $x\in \Fq^*$, note that $\xi(x)=\rho^i$ if and only if $\log_\a (x) \equiv_m i$. Hence,
    $$B_r=\{x\in \Fq^*\setminus \{1_{\Fq}\}: \log_\alpha(x-1_{\Fq}) \equiv_m \log_\rho(r-1_R) \text{ and }\log_\alpha(x) \equiv _m \log_\rho(r) \}.$$
    Hence, $B_r=C_q^m(\a;\log_{\rho}(r-1_R),\log_{\rho}(r))$ and the formula for $b$ as stated is now clear.

    If $G=R\setminus \{0_R\}$, then $G\cap (G+1_R)=\{\rho ^j : 1\le j \le m-1\}$, from which the simplified formula follows.
\end{proof}

\subsection{Recovering existing constructions of Neumaier graphs}
We recover several existing constructions of Neumaier graphs by giving four examples of applications of Theorem~\ref{thm:general}, that are equivalent to constructions due to Greaves and Koolen \cite{GK2018}, to Abiad, Castryck, De Boeck, Koolen, and Zeijlemaker \cite{ACDKZ2023}, and to Greaves and Tan \cite{GT2026}. In particular, all examples of Neumaier graphs from \cite{ACDKZ2023, GK2018} are special cases of Theorem~\ref{thm:general}.

\subsubsection*{Constructions from \cite{GK2018}}
We recover \cite[Theorems 3.3 and 3.6]{GK2018} as applications of Theorem~\ref{thm:general}.

\begin{exa}[cf.\ {\cite[Theorem 3.3]{GK2018}}]\label{exa:const18 1}
    Let $R=\F_8$, and let $\rho$ be a primitive element of $\F_8$ so that $\F_8=\F_2[\rho]$ and $\rho^3=\rho^2+1$. Let $q\equiv_{14} 1$ be a prime power, let $\a$ be a primitive element of $\Fq$, and define $\xi_1$ and $\xi_2$ as group homomorphisms $\Fq^*\to R^*$ through $\xi_1(\a)=\rho$ and $\xi_2(\a)=\rho^3$. Then $\xi_i(-1)=-1$ by Lemma~\ref{lema:-1} for $i=1,2$, so we can apply Theorem~\ref{thm:general}. For the values of the parameter $b$ (which we denote $b_1$ and $b_2$ in this case), note that by the formula given in Lemma~\ref{lema:b cyclotomy} and the first two parts of Proposition~\ref{prp:cyclo} we have
    \begin{align*}
        b_1&=2+c_q^7(\a;5,1)+c_q^7(\a;3,2)+c_q^7(\a;2,3)+c_q^7(\a;6,4)+c_q^7(\a;1,5)+c_q^7(\a;4,6)\\
        &=2+6c_q^7(\a;1,5)\\
        b_2&=2+c_q^7(\a;3,1)+c_q^7(\a;6,2)+c_q^7(\a;1,3)+c_q^7(\a;5,4)+c_q^7(\a;4,5)+c_q^7(\a;2,6)\\
        &=2+6c_q^7(\a;1,3).
    \end{align*}
    Now if $c^7_q(\a;1,5) \equiv _4 1$, then $|R|=8$ divides $b_1$ and we get a Neumaier graph $\overline\Gamma(\F_8,q,\xi_1)$. Similarly, if $c^7_q(\a;1,3)\equiv_4 1$, then $\overline \Gamma(\F_8,q,\xi_2)$ is a Neumaier graph. This example is equivalent to the construction given in \cite[Theorem 3.3]{GK2018}, when the permutations $\pi_2,\dots,\pi_t$ as in Theorems~\ref{thm:EvansMachine} and \ref{thm:general} are chosen to be identity permutations.
\end{exa}
Note that Example~\ref{exa:const18 1} uses $(\F_8,+)=(\F_2[\rho],+)$, whereas \cite[Theorem 3.3]{GK2018} uses the group $(\Z_2,+)^3$. However, considering $\F_2[\rho]$ as a 3-dimensional vector space over $\F_2$ with the coordinatization mapping $\varphi :\F_2[\rho] \to (\F_2)^3$ with respect to the basis $\{\rho,\rho^4,\rho^2\}$, in fact $\Psi_1=\log_{\rho} \circ \varphi^{-1}$ and $\Psi_2=\log_{\rho^3} \circ \varphi^{-1}$, with $\Psi_i$ from \cite[Theorem 3.3]{GK2018}.

Also note that Lemma~\ref{lema:-1} gives $q\equiv_7 1$ as a requirement, not $q\equiv_{14} 1$ as in Example~\ref{exa:const18 1}; as we will see shortly in Theorem~\ref{thm:norm}, it is possible to obtain Neumaier graphs when $R=\F_8$ and $q$ is a power of 8.

\begin{exa}[cf.\ {\cite[Theorem 3.6, Example 1]{GK2018}}]\label{exa:const18 2}
    Let $R=\F_4$, and let $\rho$ be a primitive element of $\F_4$ (so that $\F_4=\F_2[\rho]$ and $\rho^2=\rho+1$). Let $q\equiv_6 1$ be a prime power, let $\a$ be a primitive element of $\F_q$, and define $\xi:\Fq^*\to \F_4^*$ through $\xi(\a)=\rho$. Then $\xi(-1)=-1$ by Lemma~\ref{lema:-1}, so we can apply Theorem~\ref{thm:general}. For the value of $b$, by Lemma~\ref{lema:b cyclotomy} and the first or second part of Proposition~\ref{prp:cyclo} we have
    $$b=2+c_{q}^3(\a;1,2)+c_{q}^3(\a;2,1)=2+2c_{q}^3(\a;1,2).$$
    If $c_{q}^3(\a;1,2)$ is odd, then $|R|=4$ divides $b$ and $\overline\Gamma(\F_{4},q,\xi)$ is a Neumaier graph. By the third part of Proposition~\ref{prp:cyclo} we have
    $$c_q^3(\a;0,2)+c_q^3(\a;1,2)+c_q^3(\a;2,2)=1+c_q^3(\a;0,0)+c_q^3(\a;1,0)+c_q^3(\a;2,0),$$
    in which two pairs of numbers in the above equation cancel out by the first and second parts of Proposition~\ref{prp:cyclo}, so that we have $c_q^3(\a;1,2)=1+c_q^3(\a;0,0)$. It now follows that $c_q^3(\a;1,2)$ is even if and only if $c_q^3(\a;0,0)$ is odd, which by the last part of Proposition~\ref{prp:cyclo} is if and only if $2$ is a third power in $\Fq$. Thus, we get a Neumaier graph $\overline \Gamma(\F_4,q,\xi)$ if and only if $2$ is not a third power in $\Fq$. For example, if $q=7^e$ with $e$ not divisible by 3, then $\overline \Gamma (\F_{4},7^e,\xi)$ is a Neumaier graph. This example is equivalent to the construction given in \cite[Theorem 3.6 and Example 1]{GK2018}.
\end{exa}
Again, \cite[Theorem 3.6]{GK2018} uses $\Z_2^2$ instead of $\F_4$, but this can be seen to be equivalent after any choice of basis as in Example~\ref{exa:const18 1}. We would now like to note that all examples of Neumaier graphs from \cite{GK2018} are in fact covered by Theorem~\ref{thm:general}.

Also note that in Theorem~\ref{thm:norm} we see that $q$ a power of 4 (as opposed to $q\equiv_6 1$ as in Example~\ref{exa:const18 2}) also gives Neumaier graphs for $R=\F_4$.

\subsubsection*{Construction from \cite{ACDKZ2023}}

We recover \cite[Theorems 4.7 and 4.9]{ACDKZ2023} as an application of Theorem~\ref{thm:general}.
\begin{exa}[cf.~{\cite[Theorems 4.7 and 4.9]{ACDKZ2023}}]\label{exa:2023}
    Let $q=p$ be a prime, and let $\alpha$ be a primitive element of $\Fp$. Let $r\ge 3$ be an odd integer coprime to $p$ and let $R=\Z_r$. Let $\xi: \Fp^* \to \Z_r^*$ be a group homomorphism such that $\xi(-1_{\Fp})=-1_{\Z_r}$. If $r$ is a divisor of $b$ (with $b$ as in Theorem~\ref{thm:general}), then $\overline \Gamma (\Z_r,p,\xi)$ is a Neumaier graph.
\end{exa}
This example includes the construction given in \cite[Theorem 4.9]{ACDKZ2023} (where $r$ is assumed to be a prime different from $p$). In \cite{ACDKZ2023} different notation is used. The graph $\Gamma(\Z_r,p,\xi)$ from Example~\ref{exa:2023} is isomorphic to the graph $\Gamma_{pr}(a)$ of \cite[Theorem 4.7]{ACDKZ2023}, where $a\in \Z_{pr}$ is such that $a\equiv_p \a$ and $a\equiv_r \xi(\a)$. Furthermore, the set $S_{pr}(a)$ from \cite{ACDKZ2023} (which is defined as $\langle a\rangle_{\Z^*_{pr}} $) is in bijection with $\Fp^*$. Moreover, the set $S_{pr}(a) \cap (S_{pr}(a) +1)$ has cardinality equal to the parameter $\lambda$ of the edge regular graph $\Gamma_{pr}(a)$ due to \cite[Theorem 4.7]{ACDKZ2023}, and it is in bijection with $B\setminus \{0,1\}$. As a consequence, the requirement in Example~\ref{exa:2023} that $r$ is a divisor of $b$ is equivalent to the requirement of \cite[Theorem 4.9]{ACDKZ2023} that $|S_{pr}(a) \cap (S_{pr}(a)+1)|\equiv_r -2$. In particular, \cite[Theorem 4.9]{ACDKZ2023} can be seen as a special case of Theorem~\ref{thm:general}.

\subsubsection*{Construction of \cite{GT2026} with nexus 1}

Using our main result we recover the Neumaier graphs from \cite{GT2026} with nexus 1. However, we note that in \cite{GT2026}, also Neumaier graphs of other nexuses are constructed.
\begin{exa}[cf.~\cite{GT2026}]\label{exa:Tan}
    Let $q\ge q'\ge 3$ be two prime powers such that $q=1+n(q'-1)$, and let $\a$ and $\a'$ be primitive elements of $\F_q$ and $\F_{q'}$. Define $\xi:\F_q^*\to \F_{q'}^*$ through $\xi(\a)=\a'$. By Lemma~\ref{lema:-1}, the map $\xi:\F_q^*\to \F_{q'}^*$ is suitable for Theorem~\ref{thm:general} if and only if $q'$ is even or $n$ is odd. If $b=q'$ (with $b$ as in Theorem~\ref{thm:general}), then we get a Neumaier graph $\overline \Gamma(\F_{q'},q,\xi)$.
\end{exa}
This is equivalent to the restriction of \cite[Theorem 3.9]{GT2026} to the case where the nexus $n_1$ is equal to 1. Indeed, with the notation of \cite{GT2026} in this case $q_1=q'$, $q_2=q$ and $n_2=n$. The requirement $q_1n_1\equiv_2 q_2n_2$ from \cite[Theorem 3.9]{GT2026} is now equivalent to the requirement of $q'$ being even or $n$ being odd from Example~\ref{exa:Tan}. Moreover, using Lemma~\ref{lema:b cyclotomy} it can be seen that the requirement in \cite[Equation (7)]{GT2026} is equivalent to the requirement $b=q'$ of Example~\ref{exa:Tan}. In particular, all examples of \cite{GT2026} with nexus 1 are covered by Theorem~\ref{thm:general}.

\section{A construction from the norm function on finite fields}\label{sec:norm}

In this section we present our first new infinite construction of Neumaier graphs (Theorem~\ref{thm:norm}), and prove that it produces infinitely many strictly Neumaier graphs (Theorem~\ref{thm:norm infinite}). We also show that this family of Neumaier graphs includes the collinearity graphs of generalized quadrangles of $*$-Tits type (Example~\ref{exa:d=2}).

The construction involves taking (with the notation of Theorem~\ref{thm:general}) for $R$ a subfield of $\F_q$, and for $\xi$ the norm function of this extension of finite fields. Recall for a prime power $q$ and a positive integer $d$ that the field $\F_q$ is a subfield of $\F_{q^d}$, and that the \emph{norm} with respect to this extension of finite fields is defined as
$$N_q^d:\F_{q^d}\to\F_q ,\quad x \mapsto \prod_{i=0}^{d-1} x^{q^i}.$$

\begin{thm}\label{thm:norm}
    Let $q$ be a prime power and $d$ a positive integer such that $q$ is even or $d$ is odd. Then $\overline \Gamma(\F_q,q^d,N_q^d)$ is a Neumaier graph with parameters
    $$\big (b_q^d\cdot q^d\!,\ b_q^d+q^d-2,\ b_q^d-2;\ 1,\ b_q^d\big ),$$
    where
    $$b_q^d=|\{x\in \F_{q^d} : N_q^d(x-1)=N_q^d(x)-1\}|.$$
    If $b_q^d>q$, then $\overline \Gamma (\Fq,q^d,N_q^d)$ is a strictly Neumaier graph.
\end{thm}

\begin{table}[ht]
    \begin{center}
        \begin{tabular}{|c|ccccc}
        \hline
            \diagbox{$q$}{$d$} & 1 & 2 & 3 & 4 & 5\\
            \hline
            2 & 2 & 2 & 2 & 2 & 2\\
            3 & 3 & 3 & 9 & 21 & 63\\
            4 & 4 & 4 & 20 & 52 & 244\\
            5 & 5 & 5 & 20 & 125 & 575\\
            7 & 7 & 7 & 56 & 350 & 2\,254\\
            8 & 8 & 8 & 56 & 440 & 3\,992\\
            9 & 9 & 9 & 81 & 711 & 6\,723 \\
            11 & 11 & 11 & 110 & 1\,331 & 14\,762\\
            13 & 13 & 13 & 182 & 2\,145 & 28\,054\\
            16 & 16 & 16 & 272 & 3\,856 & 66\,064\\
        \end{tabular}
    \end{center}
    \caption{Values of $b^d_q$ for $q\le 16$ and $d\le 5$.}
    \label{tab:b^d_q}
\end{table}

\begin{table}[ht]
    \begin{center}
        \begin{tabular}{rr|r||rrrrr}
            $q$ & $d$ & $b_q^d$ & $v$ & $k$ & $\lambda$ & $e$ & $s$\\
            \hline
            \hline
            3 & 3 & 9 & 243 & 34 & 7 & 1 & 9 \\
            \hline
            4 & 3 & 20 & 1\,280 & 82 & 18 & 1 & 20\\
            \hline
            5 & 3 & 20 & 2\,500 & 143 & 18 & 1 & 20\\
            \hline
            4 & 4 & 52 & 13\,312 & 306 & 50 & 1 & 52\\
            \hline
            3 & 5 & 63 & 15\,309 & 304 & 61 & 1 & 63\\
            \hline
            7 & 3 & 56 & 19\,208 & 397 & 54 & 1 & 56\\
            \hline
            8 & 3 & 56 & 28\,672 & 566 & 54 & 1 & 56\\
            \hline
            9 & 3 & 81 & 59\,049 & 808 & 79 & 1 & 81\\
            \hline
            11 & 3 & 110 & 146\,410 & 1\,439 & 108 & 1 & 110\\
            \hline
            4 & 5 & 244 & 249\,856 & 1\,266 & 242 & 1 & 244\\
            \hline
            13 & 3 & 182 & 399\,854 & 2\,377 & 180 & 1 & 182\\
            \hline
            16 & 3 & 272 & 1\,114\,112 & 4\,366 & 270 & 1 & 272\\
            \hline
            3 & 7 & 549 & 1\,200\,663 & 2\,734 & 547 & 1 & 549\\
            \hline
            17 & 3 & 272 & 1\,336\,336 & 5\,183 & 270 & 1 & 272\\
            \hline
            5 & 5 & 575 & 1\,796\,875 & 3\,698 & 573 & 1 & 575\\
            \hline
            8 & 4 & 440 & 1\,802\,240 & 4\,534 & 438 & 1 & 440\\
        \end{tabular}
    \end{center}
    \caption{The parameters of some strictly Neumaier graphs from Theorem~\ref{thm:norm} with at most two million vertices.}
    \label{tab:norm}
\end{table}

See Table~\ref{tab:b^d_q} for some small values of $b^d_q$. To prove Theorem~\ref{thm:norm}, we use a special case of the Ax-Katz Theorem \cite{Katz} (also see {\cite[Theorem 0]{MorenoMoreno}} for $r=1$ and $\mu\ge 1$).

\begin{thm}[Special case of the Ax-Katz Theorem \cite{Katz}]\label{thm:Ax-Katz}
    Let $f$ be a polynomial in $d$ variables with coefficients in $\F_q$ of degree less than $d$. Then the number of zeroes of $f$ in $\F_q^d$ is divisible by $q$.
\end{thm}

\begin{lema}\label{lema:norm}
    Let $q$ be a prime power and $d$ a positive integer. Let $\psi_q^d:\F_{q^d}\to \F_q$ be defined as $x\mapsto N_q^d(x)-N_q^d(x-1)$. Then $|(\psi^d_q)^{-1}(r)|$ is divisible by $q$ for every $r \in \Fq$.
\end{lema}
\begin{proof}
    We use a proof technique from \cite{MorenoMoreno} that uses a basis $\a_1=1,\a_2,\dots,\a_d$ of $\F_{q^d}$ seen as a vector space over $\Fq$. For $x\in \F_{q^d}$, write $x=x_1+x_2\a_2+\dots+x_d\a_d$ with $x_1,\dots,x_d\in \F_q$. Using the fact that $x\mapsto x^q$ is a field automorphism of $\F_{q^d}$ fixing every element of $\Fq$, we have
    \begin{align*}
        N_q^d(x) &= \prod_{i=0}^{d-1} x^{q^i}=\prod_{i=0}^{d-1} \big(x_1+x_2\a_2+\dots+x_d\a_d \big )^{q^i}\\
        &=\prod_{i=0}^{d-1} \left (x_1^{q^i} +x_2^{q^i}\a_2^{q^i}+\dots+x_d^{q^i}\a_d^{q^i}\right )=\prod_{i=0}^{d-1}\left ( x_1 +x_2\a_2^{q^i}+\dots+x_d\a_d^{q^i}\right ).
    \end{align*}
    Write $f$ for this homogeneous $d$-variate polynomial of degree $d$ with coefficients in $\F_{q^d}$. By writing these coefficients in terms of the basis $\{\a_1=1,\a_2,\dots,\a_d\}$, we obtain homogeneous $d$-variate polynomials $f_1,\dots,f_d$ of degree $d$ with coefficients in $\Fq$ such that
    $$f=f_1+\a_2f_2+\dots+\a_df_d.$$
    Since the norm of any element of $\F_{q^d}$ is in $\F_q$, we have that
    $$f(x_1, x_2, \ldots, x_d) = f_1(x_1,x_2,\dots,x_d)$$
    for every choice of $x_1,\dots,x_d \in \F_q$.

    Now consider the polynomial
    $$g(x_1,\dots,x_d)=f_1(x_1,x_2,\dots,x_d)-f_1(x_1-1,x_2,\dots,x_d)-r.$$
    Note that
    $$g(x_1,x_2,\dots,x_d)=\psi_q^d(x_1+x_2\a_2+\dots+x_d\a_d)-r$$ for all $x_1, \ldots, x_d \in \F_q$.
    Furthermore, since $f_1$ is a homogeneous polynomial of degree $d$, the homogeneous part of degree $d$ of $f_1(x_1-1,x_2,\dots,x_d)$ is equal to $f_1(x_1,x_2,\dots,x_d)$, and so the degree of $g$ is at most $d-1$.
    
    From the Ax-Katz Theorem (Theorem~\ref{thm:Ax-Katz}) it follows that the number of zeroes of $g$ (i.e.\ the number of $x\in \F_{q^d}$ with $\psi_q^d(x)=r$) is divisible by $q$.
\end{proof}

We can now prove Theorem~\ref{thm:norm}.

\begin{proof}[Proof of Theorem~\ref{thm:norm}]
    Note that $N_q ^d$ is a multiplicative function sending 0 to 0 and 1 to 1. Furthermore, if $q$ is even then $N_q^d(-1)=N_q^d(1)=1=-1$, and if $d$ is odd then $N_q^d(-1)=(-1)^d=-1$. Next, $b_q^d=|(\psi_q^d)^{-1}(1)|$ is divisible by $q$ by Lemma~\ref{lema:norm}. The result now follows from Theorem~\ref{thm:general}, where the strictness follows from $t=b_q^d/q$; if $b_q^d>q$, then $t>1$, and the graph is a strictly Neumaier graph.
\end{proof}

We now prove that infinitely many graphs from Theorem~\ref{thm:norm} are not strongly regular. We do this by investigating the parity of the parameter $b_q^d$, in a similar way as is done in the proof of \cite[Theorem 5.1]{ACDKZ2023} (also see the proof of Lemma~\ref{lema:3}).

\begin{thm}\label{thm:norm infinite}
    Let $q$ be a prime power and $d$ a positive integer such that $q$ is even or $d$ is odd. Then $b_q^d$ is even if and only if either $q$ is even, or $q$ is odd and $d\not \equiv_{\omega(q)} 1$, where $\omega(q)$ is the order of $2$ in $\Fq^*$.
    
    Consequently, if $q$ is an odd prime power, and $d$ an odd positive integer with $d \not \equiv_{\omega(q)} 1$, then $b_q^d>q$ and so the graph $\overline \Gamma(\Fq,q^d,N_q^d)$ is a strictly Neumaier graph. In particular, infinitely many graphs from Theorem~\ref{thm:norm} are strictly Neumaier graphs.
\end{thm}
\begin{proof}
    Write $B$ for the set of elements $x$ of $\F_{q^d}$ such that $N_q^d(x-1)=N_q^d(x)-1$, so that $|B|=b_q^d$. Consider the involution $\varphi:\F_{q^d} \to \F_{q^d}$ given by $x\mapsto 1-x$. If $x\in B$, then
    $$N_q^d(\varphi(x)-1)=N_q^d(-x)=-N_q^d(x)=-N_q^d(x-1)-1=N_q^d(\varphi(x))-1,$$
    so that $\varphi(x)\in B$. Since $\varphi(\varphi(x))=x$ for all $x$, the parity of $b_q^d$ is equal to that of the number of fixed points of $\varphi$ in $B$.

    Note that $a$ is a fixed point of $\varphi$ if and only if $1=2a$. In particular, if $q$ is even, then $\varphi$ has no fixed points at all, and so $b_q^d$ is even. If $q$ is odd, then $B$ has the fixed point $1/2$ if and only if $-N_q^d(1/2)=N_q^d(1/2)-1$, which is equivalent to $N_q^d(2)=2$. Since $2\in \Fq$, we have $N_q^d(2)=2^d$. We conclude that $B$ has a (necessarily unique) fixed point if and only if $2^d=2$, which holds if and only if $d\equiv_{\omega(q)}1$. The first part of the statement follows.

    If $q$ is an odd prime power and $d$ an odd integer such that $d \not \equiv_{\omega(q)} 1$, then we have that $b_q^d \not \equiv_2 q$. It is therefore impossible that $b_q^d=q$, so we have $b_q^d>q$. The second part of the statement now follows from Theorem~\ref{thm:norm}, where for the infinitude, take for example $q=5$, for which $\omega(q)=4$, providing a strictly Neumaier graph for every $d\equiv_4 3$.
\end{proof}

Theorem~\ref{thm:norm infinite} in fact implies that for all odd $d\ge 3$ there exists some $Q(d)$ such that for every odd prime power $q> Q(d)$ we have  $b_q^d>q$ and therefore  $\overline \Gamma(\F_q,q^d,N_q^d)$ is a strictly Neumaier graph. Indeed, we can take $Q(d)=2^d$: if $q> 2^d$ is odd, then $\omega(q)>d$ and so $d\not \equiv_{\omega(q)} 1$. Table~\ref{tab:b^d_q} suggests that the inequality $b_q^d>q$ even holds under much weaker assumptions. Perhaps this is true as soon as $q\ge 3$ and $d\ge 3$. At least, using the Lang--Weil type bound from~\cite[Theorem 1.10]{Slavov}, it is possible to lower $Q(d)$  considerably to $15(d-1)^{13/3}$, and furthermore this approach also works for even $d\ge 3$ and even $q\ge 3$. A proof can be found in Appendix~\ref{app:weil_bound}, see Theorem~\ref{thm:weil_bound}.

If $q=2$ then $\overline \Gamma(\F_2,2^d , N^d_2)$ is complete bipartite and if $d=1$ then $\overline \Gamma(\Fq,q,N^1_q)$ is a square lattice graph. For $d=2$, we have the following, where
$$\Tr_q^d : \F_{q^d} \to \F_q,\quad x \mapsto \sum_{i=0}^{d-1} x^{q^i}$$
is the trace map, which is $\F_q$-linear and surjective.

\begin{exa}\label{exa:d=2}
    If $d=2$, then $b_q^d=q$. Indeed, in this case $\psi(x)=\Tr_q^2(x)-1$, so that $|\psi^{-1}(r)|=q$ for every $r\in \Fq$. In order to apply Theorem~\ref{thm:norm}, we need $q$ to be even, say $q=2^e$. Then the graph $\overline \Gamma(\F_{2^e},4^{e},N^2_{2^e})$ is a strongly regular graph with parameters $(8^e,4^e+2^e-2,2^e-2,2^e+2)$. More precisely, it is isomorphic to the collinearity graph of the generalized quadrangle of $\ast$-Tits type of order $(2^e-1,2^e+1)$ coming from the classical hyperoval (see \cite[Section 3.7.2]{VMaldeghem}). We will give a detailed explanation of this graph isomorphism shortly. For all relevant definitions, we refer to \cite[Section 7.5.2]{VMaldeghem}.
    
    Consider the projective space $\PG(3,2^e)$ consisting of points $(X:Y:Z:W)$, which partitions into its set of affine points ($W=1$) $\AG(3,2^e)$, and its points at infinity, which we identify with $\PG(2,2^e)=\{(X:Y:Z:0)\}$. Let $\a\in \F_{4^e}\setminus \F_{2^e}$ and consider the absolutely irreducible conic $\mathcal C$ in $\PG(2,2^e)$ defined by $X^2+N_{2^e}^2(\a Y+Z)=0$, which expands to $X^2+N_{2^e}^2(\a) Y^2+\Tr_{2^e}^2(\a) YZ + Z^2=0$. The nucleus of $\mathcal C$ is easily verified to be $(1:0:0:0)$ (with the notation of $\PG(3,2^e) \supseteq \PG(2,2^e)$ as described above). We now construct the generalized quadrangle; take $\AG(3,2^e)$ as the points, and take those lines of $\AG(3,2^e)$ that, when extended to a line in $\PG(3,2^e)$, intersect $\PG(2,2^e)$ in either a point of $\mathcal C$ or the nucleus. Write $\Delta$ for the collinearity graph of this generalized quadrangle.

    To show that $\Delta$ and $\overline \Gamma(\F_{2^e},4^e,N_{2^e}^2)$ are isomorphic, consider the bijection $\varphi :\AG(3,2^e) \to \F_{2^e}\times \F_{4^e}$ by $(x:y:z:1) \mapsto (x^2,\a y+z)$. We partition the connecting set $S$ of $\overline \Gamma(\F_{2^e},4^e,N_{2^e}^2)$ (note that $t=1$ with $t$ as in Theorem~\ref{thm:general}) into $2^e+2$ subsets of size $2^e-1$: $S=S_0 \sqcup \bigsqcup _{U\in \mathcal U} S_U$, with $S_0=\F_{2^e}^*\times \{0\}$, $\mathcal U$ the set of 1-dimensional vector spaces over $\F_{2^e}$ that are a subspace of $\F_{4 ^e}$, and $S_U=\{(N_{2^e}^2(u),u) : u\in U\setminus \{0\}\}$. Now if $p=(x:y:z:1)\in \AG(3,2^e)$, then the set of $S_0$-neighbors of $\varphi(p)$ in $\overline \Gamma(\F_{2^e},4^e,N^2_{2^e})$ (including itself) corresponds to the subset
    $$\{\varphi^{-1}( \varphi(p)+s):s\in S_0 \text{ or }s=(0,0)\}=\{(x':y:z:1) : x'\in \F_{2^e}\}$$
    of $\AG(3,2^e)$, which defines a line in $\AG(3,2^e)$. When extended to $\PG(3,2^e)$, this line intersects $\PG(2,2^e)$ in the nucleus $(1:0:0:0)$ of $\mathcal C$. Similarly, for $U=\langle \a y_0+ z_0\rangle $, writing $x_0=\sqrt{N_{2^e}^2(\a y_0+z_0)}$, the set
    $$\{\varphi^{-1}(\varphi(p)+s):s\in S_{U} \text{ or }s=(0,0)\}=\left \{\left(x+\lambda x_0:y+\lambda y_0: z+\lambda z_0:1\right ) : \lambda \in \F_{2 ^e}\right \}$$
    is a line in $\AG(3,2^e)$ that, when extended to $\PG(3,2^e)$, intersects $\PG(2,2^e)$ in the point $(x_0:y_0:z_0:0)$ on the conic $\mathcal C$. It now follows that $\varphi$ is indeed an isomorphism of graphs.
\end{exa}

\section{Two more constructions}\label{sec:even more constructions}

In this section, we present yet two more families of Neumaier graphs, namely Theorems~\ref{thm:F4 eps^2} and \ref{thm:zeta6}. We prove that in both cases all of the resulting Neumaier graphs are in fact strictly Neumaier graphs, and furthermore that both families are infinite. Both of the constructions of this section are applications of Theorem~\ref{thm:general}, where $R$ is either the ring of dual numbers over $\F_4$ (that is, $R=\F_4[x]/(x^2)$), or a quotient ring of the Eisenstein integers. The choice of these rings may seem rather arbitrary, but it is explained by the following theorem.
\begin{thm}\label{thm:choices for R}
    Let $\overline\Gamma (R,q,\xi)$ be a strictly Neumaier graph obtained from Theorem~\ref{thm:general} using a ring $R$ and a group homomorphism $\xi:\Fq^* \to R^*$. Suppose that $|\xi(\Fq^*)|\le 7$. Then, up to isomorphism, $R$ may be chosen to be one of the following:
    \begin{itemize}
        \item a quotient ring of the ring of integers;
        \item a finite field;
        \item the ring of dual numbers over $\F_4$;
        \item a quotient ring $\Z[\zeta_6]/(a\zeta_6+b)$ for some integers $0\le a \le b$ with $a\not \equiv_3 b$.
    \end{itemize}
\end{thm}
The first case is investigated in Example~\ref{exa:2023}, the second case is investigated in Examples~\ref{exa:const18 1} and \ref{exa:const18 2} and Theorem~\ref{thm:norm}, where we should emphasize that these investigations do not cover all possible applications of Theorem~\ref{thm:general}. In any case, the remaining two cases precisely correspond to the two new constructions presented in this section. The proof of Theorem~\ref{thm:choices for R} can be found in Section~\ref{sec:choices for R}, after the two new constructions.
\subsection{Ring of dual numbers over \texorpdfstring{$\F_4$}{}}
Here we present a construction of strictly Neumaier graphs involving the ring $R=\F_4[x]/(x^2)$ of dual numbers over $\F_4$. Note that we can equivalently write $R=\F_2[\beta,\e]$, where $\beta^2=\beta+1$ and $\e^2 =0$. After the construction, we show that this produces infinitely many strictly Neumaier graphs.

\begin{thm}\label{thm:F4 eps^2}
    Let $q$ be a prime power such that $q\equiv_{12} 1$, and let $\a$ be a primitive element of $\Fq$. Let $R= \F_2[\beta,\e]$ where $\beta^2=\beta+1$ and $\e^2=0$. Define $\xi:\Fq^* \to R^*$ through $\xi(\a)=\beta+\e$. Write $c=c_q^6(\a;2,4)$. If $c \equiv_{8} 7$, then $\overline \Gamma(\F_4[\beta,\e],q,\xi)$ is a strictly Neumaier graph with parameters $(2q(c+1),q+2c,2c;1,2(c+1))$.
\end{thm}

\begin{proof}
    We have
    \begin{align*}
        (\beta+\e)^2 &= \beta+1,\\
        (\beta+\e)^3 &= \beta \e + \e + 1 \\ 
        (\beta+\e)^4 &=\beta,\\
        (\beta+\e)^5&= \beta\e + \beta+1,\\
        (\beta+\e)^6&=1.
    \end{align*}
    In view of $q\equiv_{12} 1$ and Lemma~\ref{lema:-1}, we have that $\xi$ is well-defined and $\xi(-1)=-1$. In Lemma~\ref{lema:b cyclotomy}, the group $G$ consists of the element $\beta+\e$ and the five elements listed above, and so $G\cap(G+1)=\{\beta,\beta+1\}$, so that $b=2+c_q^6(\a;2,4)+c_q^6(\a;4,2)$. By the first or second part of Proposition~\ref{prp:cyclo}, these two cyclotomic numbers are equal to $c$ and so $b=2c+2$. Now $b$ is a multiple of $|R| = 16$ if and only if $c\equiv_8 7$. Moreover, $R$ is clearly not a field. The result now follows from Theorem~\ref{thm:general}.
\end{proof}

\begin{table}[ht]
    \begin{center}
        \begin{tabular}{r|r||rrrrr}
            $q$ & $c_q^6(\a;2,4)$ & $v$ & $k$ & $\lambda$ & $e$ & $s$\\
            \hline
            \hline
            193 & 7 & 3\,088 & 207 & 14 & 1 & 16\\
        \hline
            229 & 7 & 3\,664 & 243 & 14 & 1 & 16\\
        \hline
            277 & 7 & 4\,432 & 291 & 14 & 1 & 16\\
        \hline
            313 & 7 & 5\,008 & 327 & 14 & 1 & 16\\
        \hline
            $361=19^2$ & 15 & 11\,552 & 391 & 30 & 1 & 32\\
        \hline
            421 & 7 & 6\,736 & 435 & 14 & 1 & 16\\
        \hline
            829 & 15 & 26\,528 & 859 & 30 & 1 & 32\\
        \end{tabular}
    \end{center}
    \caption{The prime powers $q\le 1000$ with $q\equiv_{12} 1$ and $c_q^6(\a;2,4)\equiv_8 7$, together with the parameters of the corresponding strictly Neumaier graphs from Theorem~\ref{thm:F4 eps^2}.}
    \label{tab:F4 eps^2}
\end{table}

Next we prove the infinitude of Theorem~\ref{thm:F4 eps^2}, similarly as it was done in \cite[Theorem 5.20]{ACDKZ2023}.

\begin{thm}\label{thm:F4 eps^2 infinite}
    There exist infinitely many primes $p\equiv_{12} 1$ for which $c=c_p^6(\a;2,4)\equiv_8 7$. In other words, Theorem~\ref{thm:F4 eps^2} gives infinitely many strictly Neumaier graphs.
\end{thm}
\begin{proof}
    Consider $k=192\in \Eis\setminus \{0\}$ and $z=36\zeta_6+17\in \Eis$. Then $N(36\zeta_6+17)=2197$ is coprime with $192$, so in particular $k$ and $z$ are coprime. By Theorem~\ref{thm:Dirichlet} there exist infinitely many Eisenstein primes $\pi=a\zeta_6+b$ with $\pi\equiv_{192} 36 \zeta_6+17$. Let $\pi=a\zeta_6+b$ be such a prime. Now $a\equiv_{192} 36$ and $b\equiv_{192} 17$, implying $a\ne 0$, $b\ne 0 $, and $a+b\ne 0$, so that we may assume $\pi$ to be of the second type as in Proposition~\ref{prop:eis}; now $N(\pi)=a^2+ab+b^2$ is equal to a prime $p\in \Z$ with $p\equiv_3 1$. Let $x=a/2+b$ and $y=a/2$, then $p=x^2+3y^2$ and $x\equiv_{96} 35$ and $y\equiv_{96} 18$.

    In particular, $x$ is odd and $y$ is even so that $p=x^2+3y^2\equiv_4 1$ and hence $p\equiv_{12} 1$. Furthermore, $x\equiv_3 2$ and $y\equiv_3 0$, so that $36c=(x+1)^2+3y^2$ by the first part of Lemma~\ref{lema:formulas specific cyclotomic numbers}. Also, $x\equiv_{32} 3$ and $y\equiv_{32} 18$ so that $36c\equiv_{32} 4^2+3\cdot 18^2 \equiv_{32} 28$. Now $ 9c \equiv _8 7$, and so $c\equiv_8 7$, concluding the proof.
\end{proof}

\subsection{Quotient rings of the Eisenstein integers}
Here we present a construction of strictly Neumaier graphs using quotient rings of the Eisenstein integers. Afterwards, we show that this construction produces infinitely many strictly Neumaier graphs.

\begin{thm}\label{thm:zeta6}
    Let $q$ be a prime power such that $q\equiv_{12} 7$, and let $\a$ be a primitive element of $\Fq$. Let $r\ge 10$ be an integer with $r\equiv_3 1$ such that every prime $p\equiv_3 2$ divides $r$ an even number of times, and write $r=a^2 +ab+b^2$ with $0\le a \le b$ and $a\not \equiv_3 b$. Let $R = \Eis/(a\zeta_6+b)$, and define $\xi:\Fq^*\to R^*$ through $\xi(\a)=\zeta_6$. Write $c=c_q^6(\a;2,1)$. If $r$ divides $2c+2$, then $\overline \Gamma\big (\Eis/(a\zeta_6+b),q,\xi\big )$ is a strictly Neumaier graph with parameters $\big ( 2q(c+1),q+2c,2c;1,2c+2 \big )$.
\end{thm}

\begin{proof}
    Write $I\subseteq \Eis$ for the ideal generated by $a\zeta_6+b$. Then for every non-zero $x\in I$ we have $10\le r\le N(x)$. Furthermore,
    \begin{align*}
        \zeta_6^2&=\zeta_6-1,\\
        \zeta_6^3&=-1,\\
        \zeta_6^4&=-\zeta_6,\\
        \zeta_6^5&=-\zeta_6+1,\\
        \zeta_6^6&=1.
    \end{align*}
    We have $N(\zeta_6^i-\zeta_6^j)\in \{1,3,4\}$ for distinct $i,j\in \{0,\dots,5\}$, so the group $G=\langle \zeta_6\rangle _{R^*}$ consists of $6$ distinct elements of $R$. From  Lemma~\ref{lema:-1} and using $q\equiv_{12} 7$ we see that $\xi$ is well-defined and $\xi(-1)=-1$.

    Additionally, we have $N(\zeta_6^i+1-\zeta_6^j)\in \{1,3,4,7,9\}$ for every $i$ and $j$ except for $(i,j)=(2,1)$ and $(i,j)=(4,5)$ in which case the value is $0$, so that by the above $G\cap(G+1)=\{\zeta_6,-\zeta_6+1\}$. Write $b'=|B|$. By Lemma~\ref{lema:b cyclotomy} now $b'=2+c_q^6(\a;2,1)+c_q^6(\a;4,5)$. By the first or second part of Proposition~\ref{prp:cyclo}, these two cyclotomic numbers are equal to $c$. Now $b'$ is a multiple of $|R|$ (which is equal to $r$ by Proposition~\ref{prop:eis}) if and only if $r$ divides $2c+2$. Moreover, $\xi$ is not surjective since $2$ is not in the image of $ \xi$. The result now follows from Theorem~\ref{thm:general}.
\end{proof}

In Table~\ref{tab:zeta6} we show the smallest strictly Neumaier graphs coming from Theorem~\ref{thm:zeta6}. Note that for the cases where $r$ is a prime number, the quotient $\Eis/(a\zeta_6+b)$ is actually a finite field of prime order $r$. Furthermore, for $r=28$ we have $\Eis/(2\zeta_6+4)\cong \F_4 \times \F_7$ where $\zeta_6$ corresponds to $(\beta,5)$ (with $\beta\in \F_4$ such that $\beta^2=\beta+1$) and for $r=52$ we have $\Eis/(2\zeta_6+6)\cong\F_4\times \F_{13} $ where $\zeta_6$ corresponds to $(\beta,10)$. Also, note that even though Table~\ref{tab:zeta6} only shows prime numbers $q$, there are also proper prime powers to which Theorem~\ref{thm:zeta6} applies; for $q=19^3$, we have $c=189$, so that $2c+2$ is divisible by $r=19$.

\begin{table}[ht]
    \begin{center}
        \begin{tabular}{r|r|rrr||rrrrr}
        $q$ & $c_q^6(\a;2,1)$ & $r$ & $a$ & $b$ & $v$ & $k$ & $\lambda$ & $e$ & $s$\\
        \hline
        \hline
        151 & 7 & 16 & 0 & 4 & 2\,416 & 165 & 14 & 1 & 16\\
        \hline
        223 & 7 & 16 & 0 & 4 & 3\,568 & 237 & 14 & 1 & 16\\
        \hline
        283 & 7 & 16 & 0 & 4 & 4\,528 & 297 & 14 & 1 & 16\\
        \hline
        367 & 7 & 16 & 0 & 4 & 5\,872 & 381 & 14 & 1 & 16\\
        \hline
        439 & 13 & 28 & 2 & 4 & 12\,292 & 465 & 26 & 1 & 28\\
        \hline
        463 & 7 & 16 & 0 & 4 & 7\,408 & 477 & 14 & 1 & 16\\
        \hline
        499 & 13 & 28 & 2 & 4 & 13\,972 & 525 & 26 & 1 & 28\\
        \hline
        523 & 18 & 19 & 2 & 3 & 19\,874 & 559 & 36 & 1 & 38\\
        \hline
        631 & 15 & 16 & 0 & 4 & 20\,192 & 661 & 30 & 1 & 32\\
        \hline
        751 & 25 & 13 & 1 & 3 & 39\,052 & 801 & 50 &1 & 52\\
        && 52 & 2 & 6\\
        \hline
        823 & 31 & 16 & 0 & 4 & 52\,672 & 885 & 62 & 1 & 64\\
        &&64 & 0 & 8\\
        \hline
        883 & 18 & 19 & 2 & 3 & 33\,554 & 919 & 36 & 1 & 38\\
        \hline
        919 & 27 & 28 & 2 & 4 & 51\,464 & 973 & 54 & 1 & 56\\
        \hline
        991 & 27 & 28 & 2 & 4 & 55\,496 & 1\,045 & 54 & 1 & 56\\
        \end{tabular}
    \end{center}
    \caption{The tuples $(q,r,a,b)$ with $q\le 1000$ and $r=a^2+ab+b^2\ge 10$ for which Theorem~\ref{thm:zeta6} applies; those with $q\equiv_{12} 7$, $r\equiv_3 1$ such that every prime $p\equiv_3 2$ divides $r$ an even number of times, and $2c_q^6(\a;2,1)\equiv_{r} -2$.}
    \label{tab:zeta6}
\end{table}

We show that Theorem~\ref{thm:zeta6} produces infinitely many strictly Neumaier graphs for every suitable choice of $r$, similar to Theorem~\ref{thm:F4 eps^2 infinite}.

\begin{thm}\label{thm:zeta6 infinite}
    Let $r$ be a positive integer that is not divisible by 3. Then there are infinitely many primes $p\equiv_{12} 7$ for which $r$ divides $2c+2$, where $c=c_p^6(\a;2,1)$. In particular, for every integer $r$ to which Theorem~\ref{thm:zeta6} applies, it gives infinitely many strictly Neumaier graphs.
\end{thm}

To prove Theorem~\ref{thm:zeta6 infinite} we use the following lemma.

\begin{lema}\label{lema:zeta6 infinite}
    Let $r$ be a positive integer that is not divisible by 3. Then there exist integers $x$ and $y$ satisfying
    \begin{enumerate}[label=(\roman*)]
        \item $x\equiv_6 2$ and $y\equiv_6 3$,
        \item $2r \mid (x+1)^2 + 3y^2+36$,
        \item $\gcd(r,x^2+3y^2)=1$.
    \end{enumerate}
\end{lema}

\begin{proof}
Write $f(x,y):=(x+1)^{2}+3y^{2}+36$ and $g(x,y):=x^{2}+3y^{2}$, so that (ii) and
(iii) read $2r\mid f(x,y)$ and $\gcd(r,g(x,y))=1$. As $3\nmid r$ we may write
$r=2^{m}\ell_{1}^{e_{1}}\cdots\ell_{s}^{e_{s}}$ with $m\geq0$, $s\ge 0$, $e_{i}\geq1$ and
$\ell_{1}<\dots<\ell_{s}$ primes satisfying $\ell_{i}\geq5$, so that
$2r=2^{m+1}\ell_{1}^{e_{1}}\cdots\ell_{s}^{e_{s}}$. By the Chinese Remainder Theorem it suffices to produce
integers $u_{0},v_{0}$ with $u_{0}$ even, $v_{0}$ odd and
$2^{m+1} \mid f(u_{0},v_{0})$, together with integers $u_{i},v_{i}$
satisfying
\[
  \ell_i ^{e_i } \mid f(u_{i},v_{i})
  \qquad\text{and}\qquad
  \ell_i \nmid g(u_{i},v_{i})
\]
for $1\leq i\leq s$. Indeed, any $x,y\in\mathbb Z$ with $x\equiv_{2^{m+1}}u_{0}$
and $y\equiv_{2^{m+1}}v_{0}$, with $x\equiv_{3}2$ and $y\equiv_{3}0$ and with
$x\equiv_{\ell_{i}^{e_{i}}}u_{i}$ and $y\equiv_{\ell_{i}^{e_{i}}}v_{i}$ for all
$i$ then satisfy (i) and (ii), and they satisfy (iii) as well because $g(x,y)$
is odd and coprime to every $\ell_{i}$.

Fix $i$ with $1\leq i\leq s$. The homogenization of $f$ defines the conic
\[
  \mathcal C:\ (X+Z)^{2}+3Y^{2}+36Z^{2}=0
\]
of $\PG(2,\ell_{i})$, for which
$|\mathcal C|=\ell_{i}+1$ by Proposition~\ref{prop:conic count} (using that $\ell_i\ge 5$). At infinity (setting $Z=0$)
we have the equation $X^{2}+3Y^{2}=0$, which has either zero or two solutions. Consequently at least $\ell_{i}-1$ of the $\ell_{i}+1$ points of
$\mathcal C$ are affine, so $f$ has at least $\ell_{i}-1$ zeros in
$\mathbb F_{\ell_{i}}^{2}$. At most two of these zeros lie on $g=0$: if
$f(u,v)=g(u,v)=0$, then subtracting the two equations gives
$0=f(u,v)-g(u,v)=2u+37$, which determines $u$, and
then $g(u,v)=0$ determines $v$ up to sign. As $\ell_{i}-1>2$, there exists
$(u,v)\in\mathbb F_{\ell_{i}}^{2}$ with $f(u,v)=0$ and $g(u,v)\neq0$. Fix
integers $u,v$ representing such a pair. The partial derivatives
$\partial f/\partial x=2(x+1)$ and $\partial f/\partial y=6y$ do not both vanish
at $(u,v)$ modulo $\ell_{i}$, since otherwise $u\equiv_{\ell_{i}}-1$ and
$v\equiv_{\ell_{i}}0$ and hence
$f(u,v)\equiv_{\ell_{i}}36\not\equiv_{\ell_{i}}0$, contrary to the choice of
$(u,v)$. Applying Hensel's lemma (Theorem~\ref{thm:Hensel}) we obtain
$u_{i},v_{i}\in\mathbb Z$ with $u_{i}\equiv_{\ell_{i}}u$,
$v_{i}\equiv_{\ell_{i}}v$ and $f(u_{i},v_{i})\equiv_{\ell_{i}^{e_{i}}}0$, and
then also $g(u_{i},v_{i})\equiv_{\ell_{i}}g(u,v)\not\equiv_{\ell_{i}}0$, as
required.

At the prime $2$ we invoke Hensel's lemma in its general form (Theorem~\ref{thm:Hensel general}). Let $p=2$, $a=(0,1)$, $N=3$ and $k=1$; note that $f(a)=f(0,1)=40\equiv_8 0$, and $\frac{\partial f}{\partial y} (a)=6$, which has 2-adic valuation equal to $k=1$. We obtain $u_0,v_0 \in \Z$ with $u_0 \equiv _{4} 0$ and $v_0 \equiv_4 1$ (so in particular $u_0$ is even and $v_0$ is odd) and $2^{m+1} \mid f(u_0,v_0)$. This concludes the proof.
\end{proof}

We are now ready to prove Theorem~\ref{thm:zeta6 infinite}. To do so we follow a similar argument as it was done in Theorem~\ref{thm:F4 eps^2 infinite} and \cite[Theorem 5.20]{ACDKZ2023}.

\begin{proof}[Proof of Theorem~\ref{thm:zeta6 infinite}]
Let $x$ and $y$ be as in Lemma~\ref{lema:zeta6 infinite} and consider  $z=2y\zeta_6+x-y$. Note that $N(z)=(2y)^2+2y(x-y)+(x-y)^2=x^2+3y^2$ is coprime to $36r$ by the first and third parts of Lemma~\ref{lema:zeta6 infinite}. We apply Theorem~\ref{thm:Dirichlet} to obtain infinitely many Eisenstein primes $\pi=a\zeta_6+b$ with $\pi \equiv_{36r} 2y\zeta_6 + x-y$. Now $a\equiv_{36r} 2y$ and $b\equiv_{36r} x-y$. By the first part of Lemma~\ref{lema:zeta6 infinite}, we know that $a\ne 0$, $b\ne 0$, and $a+b\ne 0$, so that we may assume $\pi$ to be of the second type as in Proposition~\ref{prop:eis}; now $N(\pi)=a^2+ab+b^2$ is equal to a prime $p\in \Z$ with $p\equiv_3 1$. Note that $a$ is even, and write $\widetilde x = a/2+ b$ and $\widetilde y =a/2$. Then $p=\widetilde x^2+ 3 \widetilde y^2$ and $\widetilde x \equiv_{18 r} x$ and $\widetilde y \equiv_{18 r} y$.

In particular, by the first part of Lemma~\ref{lema:zeta6 infinite} we have that $\widetilde x$ is even, and $\widetilde y$ is odd, so that $p=\widetilde x^2+ 3 \widetilde y^2\equiv_4 3$, from which it follows that $p\equiv_{12} 7$. Furthermore $\widetilde x\equiv_3 2$ and $\widetilde y \equiv_3 0$ so that $36c=(\widetilde x+1)^2+3 \widetilde y^2$ by the first part of Lemma~\ref{lema:formulas specific cyclotomic numbers}. Since $\widetilde x \equiv_{18r} x$ and $\widetilde y \equiv_{18 r} y$ we have $36c+36 \equiv_{18 r} (x+1)^2+3y^2+36$. It follows from the first and second parts of Lemma~\ref{lema:zeta6 infinite} that the number $(x+1)^2+3y^2+ 36$ is divisible by $9$ and $2r$, and hence by $18r$. We conclude that $36c+36$ is divisible by $18r$, and hence that $2c+2$ is divisible by $r$, completing the proof.
\end{proof}
\subsection{Proof of Theorem~\ref{thm:choices for R}}\label{sec:choices for R}

To prove Theorem~\ref{thm:choices for R}, we need the following three lemmas.

\begin{lema}\label{lema:ring symmetries}
    Let $R$, $q$, $\xi$, and $b$ be as in Theorem~\ref{thm:general}. Let $\varphi$ be a field automorphism of $\Fq$, $S$ a finite ring, and $\iota:R\hookrightarrow S$ an injective ring homomorphism (so that $|R|$ is a divisor of $|S|$). Then we have the following two isomorphisms of graphs:
    \begin{align*}
    \Gamma(S,q,\iota \circ \xi \circ \varphi) &\cong \frac{|S|}{|R|} \cdot \Gamma(R,q,\xi).
    \intertext{Furthermore, if $|S|$ is a divisor of $b$, then}
    \overline \Gamma(S,q,\iota \circ\xi\circ\varphi) &\cong \overline \Gamma(R,q,\xi).
    \end{align*}
\end{lema}
\begin{proof}
    The mapping $(r,x) \mapsto \big (r,\varphi(x)\big )$ defines a graph isomorphism from $\Gamma(R,q,\xi\circ \varphi)$ to $\Gamma(R,q,\xi)$, which also maps 1-regular cocliques to 1-regular cocliques. Hence, also $\overline \Gamma(R,q,\xi\circ \varphi)$ is isomorphic to $\overline \Gamma(R,q,\xi)$ if $|R|$ divides $b$.
    
    Next, note that for every $s\in S$, the induced subgraph of $\Gamma(S,q,\iota \circ \xi)$ on the vertex set $\{(s+r,x): r\in R, x\in \Fq\} $ is isomorphic to $\Gamma(R,q,\xi)$. So, partitioning the group $(S\times \Fq,+)$ into $|S|/|R|$ cosets of the subgroup $\im(\iota) \times \F_q$, we find an isomorphism from $\Gamma(S,q,\iota\circ \xi)$ to the disjoint union of $|S|/|R|$ copies of $\Gamma(R,q,\xi)$, such that each of the 1-regular cocliques from the first part of Theorem~\ref{thm:general} is the disjoint union of the corresponding 1-regular coclique in each of the $|S|/|R|$ copies of $\Gamma(R,q,\xi)$. It is now also clear from Theorem~\ref{thm:EvansMachine} that $\overline \Gamma(S,q,\iota\circ\xi) \cong \overline \Gamma(R,q,\xi)$ for a particular choice of the permutations $\pi_2,\dots,\pi_{b/|R|}$ as in the second part of Theorem~\ref{thm:general}.
\end{proof}

\begin{lema}\label{lema:ideal}
    Let $R$, $q$, $\xi$ and $b$ be as in Theorem~\ref{thm:general}, write $m=|\xi(\Fq^*)|$, and suppose that $|R|$ is a divisor of $b$. Let $\ell=m$ if $\Char(R)=2$ and $\ell=m/2$ otherwise. Suppose that the Neumaier graph $\overline \Gamma(R,q,\xi)$ is not a complete bipartite graph. Then it 
    is isomorphic to $\overline\Gamma(\Z[x]/I,q,\eta)$, where $I$ is a non-trivial ideal of $\Z[x]$ such that the following hold:
    \begin{itemize}
        \item $x^\ell+1\in I$,
        \item $x^i+x^j+1\in I$ for some $0\le i\le j<m$,
        \item $\Char(R)=\Char(\Z[x]/I)$ (so in particular $\Char(R)\in I$).
    \end{itemize}
\end{lema}
\begin{proof}
    Let $\alpha$ be a primitive element of $\Fq$, and write $\rho=\xi(\alpha)$. Consider the evaluation morphism $\ev_\rho :\Z[x] \to R$ defined by
    $$\ev_{\rho}\left (\sum_{i=0}^d a_i x^i \right)=\sum_{i=0}^d a_i \rho^i.$$
    Write $I=\ker(\ev_\rho)$, then we get an injective morphism of rings $\iota:\Z[x]/I \hookrightarrow R$. Define a group homomorphism $\eta:\Fq^*\to (\Z[x]/I)^*$ by $\eta(\alpha)=x+I$, so that $\xi=\iota \circ \eta$ and $\eta$ satisfies the requirements of Theorem~\ref{thm:general}. By Lemma~\ref{lema:ring symmetries} we have $\overline \Gamma(R,q,\xi)\cong \overline \Gamma(\Z[x]/I,q,\eta)$.

    It is clear that if $z\in \Z$, then $\ev_\rho(z)=0$ if and only if $\Char(R)$ divides $z$, so that $\Char(R)=\Char(\Z[x]/I)$. Furthermore, since $\rho^\ell=-1$, we have that $\ev_\rho(x^\ell+1)=0$ so that $x^\ell+1\in I$.

    Let $G$ be as in Lemma~\ref{lema:b cyclotomy}. If $G\cap(G+1)=\emptyset$, then $b=2$ so $|R|=2$ and $\overline \Gamma(R,q,\xi)\cong K_{q,q}$, which is excluded in the statement. Hence we can take $r\in G\cap(G+1)=(-G) \cap (G+1)$. Write $r=-\rho^i$ and $r=\rho^j+1$ for $i,j\in \{0,1,\dots,m-1\}$, then $\rho^i+\rho^j+1=0$. Now $I$ contains the polynomial $x^i+x^j+1$, where we can assume without loss of generality that $i\le j$.
\end{proof}

\begin{lema}\label{lema:3}
    Let $R$, $q$, $\xi$ and $b$ be as in Theorem~\ref{thm:general}, and write $m=|\xi(\Fq^*)|$. If $b$ is divisible by $3$, then the following hold:
    \begin{itemize}
        \item $\Char(R)=3$,
        \item $q$ is a power of $3$,
        \item $m$ is coprime to $3$ and $m+1\in \langle 3 \rangle_{\Z_{2m}^*}$ (in particular, $m \neq 4$).
    \end{itemize}
    In particular, if $|R|$ is divisible by $3$, then in order to obtain Neumaier graphs from Theorem~\ref{thm:general} the above are necessary conditions.
\end{lema}
\begin{proof}
    Let $\xi$ and $B$ be as in Theorem~\ref{thm:general}. Consider the map $\varphi:x\mapsto 1-\frac 1x$, which defines a bijection on $\F_q^*\setminus\{0,1\}$. If $x\in B\setminus \{0,1\}$, then $\xi(x-1)=\xi(x)-1$, and after dividing by $\xi(x)$, we have $\xi\left(1-\frac1x\right)=1-\xi\left(\frac1x\right)$, which, using $\xi(-1)=-1$, is equivalent to $\xi(\varphi(x)-1)=\xi(\varphi(x))-1$, so that $\varphi(x)\in B$. Now $\varphi$ is a bijection on $B\setminus\{0,1\}$. Next, note that $\varphi(\varphi(\varphi(x)))=x$ for every $x$, so that the action of the cyclic group of order 3 through $\varphi$ partitions $B\setminus \{0,1\}$ into orbits of size 1 and 3. Note that $\{a\}$ forms an orbit of size 1 in $B\setminus \{0,1\}$ if and only if $1-\frac 1a=a$, if and only if $a$ is a root of the polynomial $x^2-x+1$. As a result, we have $b\equiv_3 2+z$, where $z$ is the number of roots in $B$ of $x^2-x+1$. Note that $0\le z \le 2$.
    
    If $b$ is divisible by 3 as in the statement, then $z=1$. Let $a \in B$ be a root of $x^2-x+1$. Consider the other root $a'$ of $x^2-x+1$, which is $a'=1-a$. Since $a \in B$, we have $\xi(a-1)=\xi(a)-1$, and since $\xi$ is multiplicative and $\xi(-1)=-1$, this implies that $\xi(a')=-\xi(a-1)=-\xi(a)+1=\xi(a'-1)+1$, so that in fact $a'\in B$. Because $z=1$ and not $z=2$, we must have $a'=a$. We obtain that $2a=1$, and from $a^2-a+1=0$ that $a=2$ and $\Fq$ has characteristic $3$, i.e.\ $q=3^e$ for some positive integer $e$. Next, from $2=-1\in B$, it follows that $1=\xi(1)=\xi(2-1)=\xi(2)-1=\xi(-1)-1=-2$, and so $R$ also has characteristic $3$.

    To show the statements concerning $m$, we apply Lemma~\ref{lema:-1}, where we note that $m=2\ell$. We get that $3^e\equiv_{2m} m+1$, so in particular $3$ is a unit in $\Z_{2m}^*$, and the subgroup of $\Z_{2m}^*$ generated by $3$ contains $m+1$.
\end{proof}

In particular, if $R=\F_3$, then $q$ has to be a power of $3$, compare \cite[Section 5.1]{ACDKZ2023}, of which the argument was the basis for Lemma~\ref{lema:3}. Since only prime fields $\Fq$ were considered in \cite{ACDKZ2023} (see Example~\ref{exa:2023}), they did not find any Neumaier graphs this way \cite[Remark 5.3]{ACDKZ2023}. However, Theorem~\ref{thm:norm} does give new examples using $R=\F_3$.

We are now ready to prove Theorem~\ref{thm:choices for R}.

\begin{proof}[Proof of Theorem~\ref{thm:choices for R}]
Let $\overline\Gamma(R,q,\xi)$ be a strictly Neumaier graph from Theorem~\ref{thm:general}. Let $\a$ be a primitive element of $\F_q$, and write $\rho=\xi(\a)$. Let $\ell$ be the smallest positive integer such that $\rho^\ell=-1_R$, and let $m=|\xi(\Fq^*)|\le 7$ be the order of $\rho$ in $R^*$. We distinguish two cases.
\begin{description}
    \item[$\Char(R)=2$:] Now $m=\ell$, and by Lemma~\ref{lema:ideal} we have that $\overline \Gamma(R,q,\xi)$ is isomorphic to a graph $\overline \Gamma(\Z[x]/I',q,\eta)$ for some ideal $I'$ of $\Z[x]$ containing the number $2$ and the polynomials $x^m+1$ and $x^i+x^j+1$ for some $i,j\in \{0,\dots,m-1\}$. Note that as rings $\Z[x]/I'\cong \F_2[x]/I$, with $I$ an ideal of $\F_2[x]$ containing $x^m+1$ and $x^i+x^j+1$ for some $0\le i \le j < m \le 7$. Now $I$ contains the polynomial $f_{i,j}^m=\gcd(x^m+1,x^i+x^j+1)$. If $f_{i,j}^m=1$, then $I=(1)$, contradicting $\Char(\F_2[x]/I)=2$ from Lemma~\ref{lema:ideal}. Otherwise, if $f_{i,j}^m$ is irreducible, then either $I=(1)$ or $I=(f_{i,j}^m)$, where in the first case we have a contradiction as above, and in the second case $\F_2[x]/I$ is a finite field. The only remaining case is where $i=2$, $j=4$, and $m=6$, where $f_{2,4}^6=x^4+x^2+1=(x^2+x+1)^2$. In this case either $I=(1)$ or $I=(x^2+x+1)$ (leading to the same conclusions as above), or $I=(f_{2,4}^6)$, in which case $\F_2[x]/I$ is isomorphic to the ring of dual numbers over $\F_4$.

    \item[$\Char(R)\ge 3$:] Now $\ell=m/2$, so $\ell\le 3$. We apply Lemma~\ref{lema:ideal}: $\overline \Gamma(R,q,\xi)$ is isomorphic to a graph $\overline{\Gamma}(\Z[x]/I,q,\eta)$, with $I$ an ideal containing $I_{ i, j}^\ell$ for some $1\le \ell \le 3$ and $0\le i \le j <2\ell$, where $I_{i, j}^\ell$ is the ideal generated by $x^\ell+1$ and $x^i+x^j+1$. Similar as above, if $I$ contains the number 1 or the number 2, then we obtain a contradiction with $\Char(\Z[x]/I)\ge 3$ from Lemma~\ref{lema:ideal}. If instead $I$ contains a monic linear polynomial, then $\Z[x]/I$ is isomorphic to a quotient ring of the ring of integers. Therefore, we only consider the cases where $I_{i,j}^\ell$ does not contain 1, 2 or a monic linear polynomial. In particular, we can exclude the case $\ell=1$, so that $m\in \{4,6\}$. Now by Lemma~\ref{lema:3}, it is impossible for $I$ to contain the number 3, so we can exclude those cases as well. There is only one case left with $1\le \ell \le 3$ and $0\le i \le j < 2\ell$, namely $\ell=3$, $i=2$, and $j=4$ (see Tables~\ref{tab:I^2_{i,j}} and \ref{tab:I^3_{i,j}}).

    \begin{table}[ht]
        \begin{center}
            \begin{tabular}{|c|cccc|}
            \hline
            \diagbox{$i$}{$j$} & 0 & 1 & 2 & 3\\
            \hline
            0 & 3 & $x+2$ & 1 & $x-2$\\
            1 & - &$x-2$ & 1 & 1\\
            2 & - & - & 1 & 1\\
            3 & - & - & - & $x+2$ \\
            \hline
            \end{tabular}
        \end{center}
        \caption{Certain key elements of ideals $I_{i,j}^2$, see the proof of Theorem~\ref{thm:choices for R}. The cases where $i>j$ are marked with a ``-''. }
        \label{tab:I^2_{i,j}}
    \end{table}

    \begin{table}[ht]
        \begin{center}
        \begin{tabular}{|c|cccccc|}
            \hline
            \diagbox{$i$}{$j$} & 0 & 1 & 2 & 3 & 4 & 5\\
            \hline
            0 & 3 & $x+2$ & $x+4$ & 1 & $x-2$ & $x+4$ \\
            1 & - & $x+4$ & 2 & 1 & 1 & 2  \\
            2 & - & - & $x-2$ & 1 & $\ast$ & 1 \\
            3 & - & - & - & 1 & 1 & 1 \\
            4 & - & - & - & - & $x+4$ & 2 \\
            5 & - & - & - & - & - & $x+2$ \\
            \hline
        \end{tabular}
        \end{center}
        \caption{Certain key elements of ideals $I_{i,j}^3$, see the proof of Theorem~\ref{thm:choices for R}. The cases where $i>j$ are marked with a ``-''. The interesting case where $i=2$ and $j=4$ is marked with a ``$\ast$''.}
        \label{tab:I^3_{i,j}}
    \end{table}

    In this case, we have $I^3_{2,4}=(x^2-x+1)$, and $\Z[x]/I^3_{2,4}\cong\Z[\zeta_6]$. Since $I^3_{2,4} \subseteq I$, now $\Z[x]/I$ is isomorphic to $\Z[\zeta_6]/J$ for some ideal $J$ of $\Z[\zeta_6]$. By Proposition~\ref{prop:eis} we have $J=(a\zeta_6 +b)$ for some $0\le a$ and $0\le b$. If $a>b$, then the conjugated ideal $\overline J=(a(1-\zeta_6)+b)=(a(\zeta_6-\zeta_6^2)+b\zeta_6)=(b\zeta_6+a)$, so we can assume that $a\le b$. Since $R$ is supposed to be finite, it has finite characteristic, and so $J\ne (0)$. Now by Proposition~\ref{prop:eis} we have $|\Z[\zeta_6]/J|=a^2+ab+b^2$. To complete the proof, suppose that $a\equiv_3 b$. Then $|\Z[x]/I|$ is divisible by $3$, so that by Lemma~\ref{lema:3}, the parameter $m$, which in this case is equal to the order of $x+I$ in $(\Z[x]/I)^*$, must be coprime to 3. Since $x^6-1 \in I^3_{2,4}$ it follows that $ m \mid 2$, and that $I$ contains the monic linear polynomial $x+1$ or $x-1$. This concludes the proof. \qedhere
    \end{description}
\end{proof}

	\subsection*{Acknowledgements}
    \par Aida Abiad is supported by the Dutch Research Council (NWO) through the grant VI.Vidi.213.085. Wouter Castryck is supported by the European Research Council under grant agreement 101020788 (ISOCRYPT), by the Research Council KU Leuven under grant C14/24/099 and by CyberSecurity Research Flanders with reference number VOEWICS02. Maarten De Boeck was partially supported by the Croatian Science Foundation under the project 5713. Thijs van Veluw was partially supported by the Fund Professor Frans Wuytack. 
    
    The authors thank Tim Beyne and Sergey Goryainov for inspiring discussions on the topic. Part of this work was done while the first author was visiting the Simons Institute for the Theory of Computing at UC Berkeley. 

	\subsection*{Statements and Declarations}
	The authors have no relevant financial or non-financial interests to disclose. Claude Opus 5 has been used to polish the writing of the proof of Lemma~\ref{lema:zeta6 infinite} and the corresponding section in the Preliminaries. 
    It also produced a step in the proof of Theorem~\ref{thm:weil_bound}(concretely, the first sentence of the last paragraph). Apart from this, AI has not been used for any mathematical reasoning.
    GPT-5.6 (OpenAI; accessed September 2026) has been used for assistance in proofreading. 


\newpage
\appendix
\section{Appendix} \label{app:weil_bound}

In Theorem~\ref{thm:weil_bound} we show how to use Lang--Weil type bounds to significantly lower the parameter $Q(d)$, see Section~\ref{sec:norm}.

\begin{thm} \label{thm:weil_bound}
  If $d \geq 3$ and $q > 15 (d-1)^{13/3}$ then $b_q^d > q$. In particular, if moreover $q$ is even or $d$ is odd, then $\overline{\Gamma}(\F_q,q^d,N_q^d)$ is a strictly Neumaier graph.
\end{thm}

\begin{proof}
We build on the proof of Lemma~\ref{lema:norm} and keep the same notation. If $q > 15 (d-1)^{13/3}$ then the number of zeroes of the polynomial $g$ is at least
\[
  q^{d-1} - (d-2)(d-3)q^{d-3/2} - (d - 2/5)q^{d-2}
\]
by~\cite[Theorem 1.10]{Slavov}, provided that $g$ is absolutely irreducible, i.e., irreducible as an element of $\overline{\F}_q[x_1, \ldots, x_d]$ where $\overline{\F}_q$ denotes an algebraic closure of $\F_q$. One verifies:
\[
  q^{d-1} - (d-2)(d-3)q^{d-3/2} - (d - 2/5)q^{d-2} - q > q^{d-1} - 3(d-1)^2q^{d - 3/2} > 0,
\]
using $d \geq 3$ and the lower bound on $q$.

So it is enough to show that $g$ is absolutely irreducible. In the proof of Lemma~\ref{lema:norm} we observed that the degree-$d$ polynomials $f_2, \ldots, f_d$ evaluate to $0$ at all elements of $\F_q^d$. Since $q > d$ this means that $f_2 = \ldots = f_d = 0$, see~\cite[Lemma 7.40]{LidlNiederreiter}. Therefore
\[ f_1 = f = \prod_{i=0}^{d-1}\left ( x_1 +\a_2^{q^i}x_2+\dots+\a_d^{q^i}x_d\right ), \]
really as an identity of polynomials. The $d$ linear factors on the right are polynomials over $\F_{q^d}$ and they are $\F_{q^d}$-linearly independent: to see this, one can assume that $\alpha_1 = 1, \alpha_2, \alpha_3 = \alpha_2^2, \ldots, \alpha_d = \alpha_2^{d-1}$ is a power basis and then the matrix
\[ \left( \alpha_j^{q^i} \right)_{1 \leq j \leq d, 0 \leq i \leq d - 1} = \left( (\alpha_2^{q^i} )^{j-1} \right)_{1 \leq j \leq d, 0 \leq i \leq d - 1} \] 
is the Vandermonde matrix associated with the roots of the minimal polynomial of $\alpha_2$ over $\F_q$, which are all distinct.

Thus, an $\F_{q^d}$-linear change of variables transforms the polynomial $f_1(x_1, x_2, \ldots, x_d)$ into $x_1x_2 \cdots x_d$ and under this transformation $f_1(x_1 - 1, x_2, \ldots, x_d)$ transforms into the polynomial $(x_1 - 1)(x_2 - 1) \cdots (x_d - 1)$. It is therefore sufficient to prove that the polynomial $x_1 x_2 \cdots x_d - (x_1 - 1)(x_2-1) \cdots (x_d - 1) - r$
is absolutely irreducible. For this it is enough that its homogeneous part of highest degree is absolutely irreducible; call this part $h$. Since $h$ has degree $1$ in $x_1$, any candidate-factorization $h = h_1 h_2$ must be such that one factor, say $h_1$, involves $x_1$ and the other factor does not. 
Then upon writing $h = h' x_1 + h''$ for polynomials $h', h''$ not involving $x_1$, we have $h_2 \mid h'$ and $h_2 \mid h''$. But one explicitly checks that $h'$ has no common factors with $h'' = x_2 x_3 \cdots x_d$, so $h_2$ must be a unit, i.e., the factorization is necessarily trivial.
\end{proof}


\end{document}